\pdfoutput=1

\documentclass[10pt]{article}

\usepackage[hidelinks,colorlinks = true, citecolor = {teal}, breaklinks, backref=section, pagebackref=true, linktoc=all]{hyperref}
\renewcommand*{\backref}[1]{}
\renewcommand*{\backrefalt}[4]{%
 \ifcase #1
   (Not cited.)%
 \or 
   \pdftooltip[mathstyle=\displaystyle]{{\color{xkcdGrey}$\hookleftarrow$}}{Appears on} \color{white!30!black}page~#2.%
 \else
   \pdftooltip[mathstyle=\displaystyle]{{\color{xkcdGrey}$\hookleftarrow$}}{Appears on} \color{white!30!black}pages~#2.%
 \fi}
 \usepackage{pdfcomment}

\usepackage{xcolor}

\usepackage[numbers]{natbib} % more citation tools

\usepackage[linesnumbered,lined,commentsnumbered,noend]{algorithm2e}
\SetKwComment{Comment}{\qquad/* }{ */}
\RestyleAlgo{ruled}
\DontPrintSemicolon
\SetAlgoSkip{}
\SetKwInOut{Input}{Input\hspace{-0.5em}}%\SetKwInOut{Output}{Output}

\let\oldnl\nl
\newcommand{\nlnonumber}{\renewcommand{\nl}{\let\nl\oldnl}}
\SetAlCapFnt{\footnotesize}
\usepackage[sfmath]{kpfonts}

\usepackage{charter} % Use the Charter font for the document text

\usepackage[utf8]{inputenc}
\usepackage[T1]{fontenc}
\usepackage{enumitem} 	% Enumerations
\usepackage{xkcdcolors}

\usepackage{geometry}
\usepackage{setspace}
\usepackage{parskip} % Adjust paragraph spacing
\usepackage{graphicx}     % Required for inserting images
\usepackage{tcolorbox}    % Box for optimisation problems
\tcbuselibrary{breakable}
\usepackage{subcaption} % Subplots
\usepackage{wrapfig}	% Figure wraps 
\usepackage[font=footnotesize]{caption}
\usepackage{booktabs}	% Nice rules in tables
\usepackage{multirow}
\usepackage{multicol}
\usepackage{orcidlink}	% orcids
\usepackage[color = xkcdWheat, bordercolor = xkcdWheat, textsize=tiny, textwidth=14mm]{todonotes}
\usepackage{soul}	% strikethrough text and highlighting
\setstcolor{xkcdRed}

\usepackage{amsfonts}
\usepackage{amsmath}
\usepackage{amssymb,amsthm}
\usepackage{mathtools} % Many math decorators
\usepackage{commath} % Differential operators and brackets
\usepackage{xfrac} % Nice fractions, more modern
\usepackage{stmaryrd} % Additional symbols from the AMS, including the double brackets
\usepackage{dsfont}  % Additional font
\usepackage{siunitx} % units  [load b4 arydshln]
\usepackage{arydshln} % Dashed matrices
\usepackage{accents}	% interior circle
\usepackage[toc,page]{appendix}
\AtBeginEnvironment{appendices}{\crefalias{section}{appendix}}

\usepackage[capitalise]{cleveref}	% Nice formatting of cross references
\crefformat{equation}{#2(#1)#3}
\crefrangeformat{equation}{#3#1#4 to #5#2#6}
\crefmultiformat{equation}{#2(#1)#3}{ and~#2(#1)#3}{, #2(#1)#3}{, and~#2(#1)#3}
\Crefname{subsection}{Subsection}{Subsections}

\theoremstyle{definition}

\theoremstyle{remark}
\newtheorem{remark}{Remark}

\theoremstyle{definition}
\newtheorem{theorem}{Theorem}%[section]
\newtheorem{lemma}{Lemma}
\newtheorem{corollary}{Corollary}[theorem]

\theoremstyle{definition}

\makeatletter
\renewenvironment{proof}[1][\proofname]{\par
  \pushQED{\qed}%
  \normalfont \topsep0pt\relax %\topsep6\p@\@plus6\p@\relax
  \trivlist
  \item[\hskip\labelsep
        \itshape
    #1\@addpunct{.}]\ignorespaces
}{%
  \popQED\endtrivlist\@endpefalse
}
\makeatother

\newtcolorbox{highlighter}{colback=xkcdGoldenrod, boxrule=0pt, sharp corners, boxsep=0pt, left=\fboxsep, right=\fboxsep, parbox=false, breakable}

\newcommand{\embh}[1]{%
\IfDarkModeTF{{\color{white!80!\thepagecolor}\textbf{#1}}}{{\color{black!80}\textbf{#1}}}%
}

\definecolor{algcolor}{rgb}{0.97,0.97,0.97} % Very light
\definecolor{IMAcolor}{RGB}{187,39,113} % IMA color
\hypersetup{linkcolor = IMAcolor}
\colorlet{TocColor}{white!10!black}

\newcommand{\R}{\mathbb{R}}		% Real numbers
\newcommand{\Z}{\mathbb{Z}}		% Integer numbers
\newcommand{\C}{\mathbb{C}}		% Complex numbers
\newcommand{\llb}{\llbracket}
\newcommand{\rrb}{\rrbracket}
\newcommand{\Ones}[1][n]{\mathds{1}_{#1}}		% Array of ones
\newcommand{\Li}{\operatorname{Li}_2}

\allowdisplaybreaks			% Break pages for aligned equations
\usepackage{stackengine,scalerel}
\usepackage{dashbox}

\newcommand{\disc}{\raisebox{.05em}{$\bigcirc$}}
\newcommand{\smalldisc}{\scaleobj{0.9}{\newmoon}}

\usepackage[nointegrals]{wasysym} % circles
\usepackage{fontawesome} % cube

\title{The Newtonian kernel at the intersection of two discs}
\author{%
    Andrés Miniguano--Trujillo \orcidlink{0000-0002-0877-628X}%
    \thanks{School of Mathematics and Maxwell Institute for Mathematical Sciences, The University of Edinburgh, Edinburgh, United Kingdom
    (\texttt{Andres.Miniguano-Trujillo@ed.ac.uk} \orcidlink{0000-0002-0877-628X})}
    }

\date{}

\begin{document}
\spacing{1.213}
\maketitle

\vspace{-1.5\baselineskip}

\begin{abstract}
%	In this brief note, we derive an exact formula for the action of the two-dimensional Newtonian potential against the indicator function of the intersection of two discs. This quantity arises naturally when discretising the potential in nonlocal phase-separation models, providing exact error bounds for numerical approximations.
	
We present an exact, closed-form expression for the Newtonian potential of the characteristic function associated with two overlapping discs in the plane. This setting naturally arises when discretising nonlocal interaction terms present in models of phase separation, aggregation dynamics, and quantum systems. We characterise the convolution integral
%, which incorporates dilogarithmic terms, 
as a piecewise function on the distance between the disc centres, with transitions dictated by the geometry of the overlapping region. Additionally, we derive detailed asymptotic expansions for the small-overlap regime, which allows us to provide stable double-precision codes.
\end{abstract}

\vspace{-0.5\baselineskip}
{\footnotesize
\textbf{keywords:}
Newtonian potential {\textbullet} asymptotic approximations {\textbullet} numerical integration
}

%\begin{spacing}{1.1}
%{
%  \hypersetup{linkcolor=TocColor}
%    \tableofcontents    
%}
%\end{spacing}

% ––––––––––––––––––––––––––––––––––––––––––––––––––––––––––––––––––––––––––––––––– %
\section{Introduction}\label{ch:Intro}

Many contemporary models in physics, biology, and materials science involve nonlocal self-interac\-tions in which an evolving density is coupled to itself through the solution of an intermediate Poisson equation, typically realised via convolution with the Newtonian potential. In quantum settings, such nonlocalities arise in Schrödinger-Poisson and Schrödinger-Newton models of self-gravitating matter \cite{Bonheure2021a,Cassani2021a,Harrison2002a}, in gravity models based on Bose-Einstein condensates \cite{Girelli2008a,Bettoni2014a}, and in studies of dipolar quantum gases \cite{Smith2023a,Bao2010a}. Convolution-type equations govern a wide range of phenomena, including viscous and inviscid aggregation with Newtonian interaction \cite{Cozzi2017a,DiFrancesco2021a,Bertozzi2012a}, vortex-density dynamics in type-II superconductors \cite{Masmoudi2005a,Weinan1994a}, chemotactic collapse and pattern formation within Keller–Segel models \cite{Calvez2007a,Perthame2015,He2025a}, and the long-time behaviour of interacting species \cite{Carrillo2019a}. In materials science and biology, nonlocal interactions based on the same Newtonian potential play a central role in phase‑separation dynamics, notably in variants of the Cahn--Hilliard equation \cite{Giacomin1996,AGGP,Gal2017,Poiatti2024}, and appear in tumour‑growth models involving chemotaxis and active transport \cite{Garcke2016,Fornoni2023}.

Despite significant advances in temporal discretisation and operator--splitting, the numerical treatment of singular kernel convolutions frequently relies on spectral and pseudospectral codes over a box \cite{Tsubota2017a,Smith2023a,Bao2010a,Harrison2002a,AMT_Th_2024}. The present work furnishes an exact, closed‑form expression for the convolution of the Newtonian kernel with the indicator of two overlapping discs. This configuration provides one of the few nontrivial benchmark cases in which such an interaction can be computed analytically, and it serves as a key building block in a companion work, where it is leveraged to design high-order methods in circular domains.

Let \(K\) denote the Newtonian kernel
\begin{equation}\label{eq:Newtonian_potential}
	K: \R^2 \setminus\{0\} \ni x  \longmapsto  \frac{1}{2\pi} \log \|x\| = \frac{1}{4\pi} \log ( x_1^2 + x_2^2 ) \in \R.
\end{equation}
By definition, \cref{eq:Newtonian_potential} is radially symmetric. Hence, we will often use the shorthand \( k: \R \setminus \{0\} \ni r  \longmapsto  \frac{1}{4\pi} \log r^2 \in \R\) so that \(K(x) = k(r)\), where \(r = \|x\|\) denotes the Euclidean norm\footnote{Unless otherwise stated, we will from now on assume that every normed quantity is defined in terms of the Euclidean norm \(
	\| \cdot \|: \R^n  \ni x  \longmapsto  \Big( \sum\limits_{i=1}^n x_i \Big)^{ \sfrac 1 2 } \hspace{-0.2em} \in \R_{\geq 0}
\).} of \(x\). 

The Newtonian kernel is a differentiable function, and one may show that \(K\) and its gradient are locally integrable in \( \R^2\), with \(K \in W_{\text{loc}}^{1,1}(\R^2)\), see \cite[Lemma 18.2.1]{AMT_Th_2024}. Additionally, \(-K\) is the fundamental solution of the Poisson equation in \(\R^2\); i.e., for \( f \in C^2_c (\R^2)\), the convolution \( v = K \star f\)  satisfies \( -\Delta v = f\). Formally, this may be written as \( -\Delta K = \delta_0\). For a detailed discussion, we refer the reader to \cite[\S 2.2]{Evans_2010} and \cite[\S2.4 \& Chapter 4]{Gilbarg2001}. The convolution operator \(\star\) defines the \emph{Newtonian potential}
\begin{equation}\label{eq:Newt_Pot}
	K \star f \, (x) = \int\limits_{ \R^2 } K(x-y) f(y) \dif y.
\end{equation}
The measurability of \(K\) ensures that \cref{eq:Newt_Pot} is well defined for all \( x \in \R^2\). In what follows, we focus on computing \cref{eq:Newt_Pot} for a particular class of indicator functions associated with subsets of the closed unit ball with centre at the origin. Given \(\Omega \subseteq \R^2\), its indicator function \(\Ones[\Omega](x)\) takes the value of \(1\) whenever \( x\in \Omega\) and \(0\) otherwise. 

Throughout this note, we will use a graphical notation to denote certain sets. In particular, let \( \varepsilon \in (0,\sfrac 1 2] \), then we define the unit disc as the closed unit ball \( \disc \coloneqq B[0;1]\) and the intersection of \(\disc\) with the \(\varepsilon\)-disc centred at \(x\) by \( \raisebox{.05em}{\smalldisc} \coloneqq B[0;1] \cap B[x; \varepsilon] \). A comprehensive list of symbols used throughout this note is given in \cref{tb:Symbols}.

We now state the main result of this note:
\begin{theorem}\label{th:Main}
	Let \( E_\varepsilon  \coloneqq K \star \Ones[ \raisebox{.05em}{\smalldisc} ] \) and let \( x \in \R^2 \) with norm \( a \coloneqq \|x\| \). 
	Then, the value of \( E_\varepsilon(x) \) depends on the distance of \(x\) to the origin, and it is given as follows:
	\[
		E_\varepsilon (x) = \frac{1}{2\pi} \int\limits_{ x - \smalldisc } \log |y| \dif y
		=
		\frac{1}{4}
		\begin{cases}
			\varepsilon^2 (\log \varepsilon^2 - 1) 
					%& \text{if } x \in B[0;1-\varepsilon],
					%& \text{if } \|x\| \in [0,1-\varepsilon],
					& \text{if } a \in [0,1-\varepsilon],
			\\
			\dfrac{1}{\pi} (\pi - \varphi) \varepsilon^2 ( \log \varepsilon^2 - 1) 
			%+ 8 M_\varepsilon \big(\|x\|,\varphi\big)  
			+ 8 F_\varepsilon (a,\varphi)
			& \text{if } %\operatorname{dist}\big(x, \partial B(0;1) \big) \in [0,1-\varepsilon),
			%\|x\| \in (1-\varepsilon, 1+\varepsilon),
			a \in (1-\varepsilon, 1+\varepsilon),
			\\
			0 & \text{if } %\|x\| \in [1+\varepsilon, \infty).
			a \in [1+\varepsilon, \infty),
		\end{cases}
	\]
	where the intersection angle \( \varphi = \varphi_\varepsilon (a) \in [0, \pi] \) is defined by
    \[
        \varphi \coloneqq \arccos \frac{1 - a^2 - \varepsilon^2}{2 a \varepsilon}
        \qquad
        \forall a \in [1-\varepsilon,1+\varepsilon]
        .
    \]
    
    The function \( F_\varepsilon: D \to \R\), defined on the set \( D = \cbr{ (a, \varphi) : a \in [1-\varepsilon, 1+\varepsilon],\ \varphi = \varphi_\varepsilon (a) } \), is continuous and given by the following piecewise formula, where for \( a \geq 1\) we introduce \( \alpha \coloneqq \arcsin \sfrac{1}{a} \); then
	\begin{equation}
	\label{eq:Fun_M}
		F_\varepsilon (a, \varphi)
		=
		\frac{1}{8\pi}
		\begin{dcases*}
			G\del{a; \Phi(\varphi)} - \pi (1-a^2) 
									& for \( 1-\varepsilon \leq a < 1\),
			\\
			G\del{1; \Phi(\varphi)} 		& for \(a = 1\),
			\\
			G\del{a; \Phi(\varphi)} - 2\pi \log a - 2G\del{a; (\sfrac 1 4) (\pi + 2\alpha) } 
									& for \( 1 < a < \sqrt{1 + \varepsilon^2} \),
			\\
			- 2\pi \log a - G\del{a; \Phi(\varphi + \pi)} 
									& for \( \sqrt{1 + \varepsilon^2} \leq a \leq 1+\varepsilon  \).
		\end{dcases*}
	\end{equation}
	Here, \( \Phi: [0,2\pi] \to [0, \sfrac \pi 2] \) is the composition
	\[
		s(\theta) \coloneqq -a \cos \theta + \sqrt{ 1 - a^2 \sin^2 \theta },
		\qquad
		L(\theta) \coloneqq \frac{s^2 (\theta) - 1 - a^2}{2a},
		\qquad
		\Phi(\theta) \coloneqq \frac{1}{2} \arccos L(\theta);
	\]
	and \(G: [1-\varepsilon, 1+\varepsilon] \times [0,2\pi] \to \R\) is the continuous function
	\begin{align*}
		G(a;\phi)
		&\coloneqq
		2 \Im \big( \Li (-a e^{2i\phi}) \big)
		+ (1-a^2) \left[ 2 \phi - \arctan\left(\frac{ a \sin 2\phi}{1 + a \cos 2\phi} \right)  \right]
		%+ (1-a^2) \del{ 2 \phi - \arg\left( 1 + a e^{2i\phi}  \right)  }
		\\
		&\qquad
		+ a\big( 2- \log(1+ a^2 +2a\cos 2\phi)  \big)\sin 2\phi,
	\end{align*}
	where \( \Li : \C \to \C\) denotes the dilogarithm on a continuous branch.
\end{theorem}

\begin{remark}
	The domain of \(\Phi\) is fully characterised on whether \(x \in \disc\). In particular,
	\[
		\operatorname{dom} s = 
		\begin{dcases*}
			[0,2\pi] & if \( a \leq 1\),
			\\
			[-\alpha,\alpha] + \pi \, \Z & if \( a > 1 \),
		\end{dcases*}
		\qquad\text{and}\qquad
		\operatorname{ran} s = 
		\begin{dcases*}
			[1-a,1+a] & if \( a \leq 1\),
			\\
			\sbr{-\sqrt{a^2 -1},1+a} & if \( a > 1 \).
		\end{dcases*}
	\]
	In either case, \(\operatorname{ran} L = [-1,1]\), so \(\Phi\) is well defined on \(\operatorname{dom} s\).

	Moreover, the intersection angle \(\varphi\) satisfies	
	\begin{equation}\label{eq:Intersection_Angle_Cos}
	L(\varphi) =\begin{cases}
		\dfrac{ \varepsilon^2  - 1 - a^2}{ 2 a } & \text{if } a^2 \leq 1 + \varepsilon^2,
		\\[1em]
		\dfrac{ (a^2 - 1)^2 }{2 a \varepsilon^2} - \dfrac{1+a^2}{2a} & \text{if } a^2 > 1 + \varepsilon^2.
	\end{cases}
	\end{equation}
\end{remark}

\begin{remark}
	At the boundaries of the intermediate region where \(F_\varepsilon\) is defined, the intersection angle takes the values \( \varphi_{\varepsilon}(1-\varepsilon) = 0\) and \( \varphi_{\varepsilon}(1+\varepsilon) = \pi\). The continuity of \(E_\varepsilon\) across the transitions at \(\{1\pm \varepsilon\}\) implies that the function \(F_\varepsilon\) vanishes at these endpoints.
\end{remark}

Although \(\Ones[\raisebox{.05em}{\smalldisc}]\) is a simple indicator function, its Newtonian potential has a surprisingly rich structure. The reader must be cautious with these expressions, since other simple indicators might suggest an otherwise simpler structure. For example, the Newtonian potential of \(\Ones[\disc]\) is given by the simple formula\footnote{Relevant notebook in repository:
\texttt{\color{xkcdOceanBlue}1 - Disc Integral.ipynb} \newline
An alternative derivation of \cref{eq:Conv_Disc} is provided using a log-cosine integral, and a numerical test of the formula is contrasted against SciPy quadrature.}
\begin{equation}\label{eq:Conv_Disc}
	K \star \Ones[\disc]  (x) = \frac{1}{4} \del{ \|x\|^2 -1 }.
\end{equation}
Indeed, interpreting \cref{eq:Newt_Pot} as a solution of Poisson’s equation quickly yields \cref{eq:Conv_Disc}. Since radial symmetry suggests a radial ansatz \(u(r) = J(x)\), we may write \( \Delta J = 1\) in \(\disc\) in polar coordinates as \( u''(r) + r^{-1} u'(r) + r^{-2} \frac{\partial J}{\partial \theta} = 1 \), where the radial term vanishes. The resulting ordinary differential equation of Euler-Cauchy type \( r u'' + u' = r\) has the solution \( u(r) = (\sfrac{1}{4}) r^2 + c_1 \log r + c_2 \). Since \(u\) is the Newtonian potential of a bounded source, we have \(u \in C^{1,\alpha}(\disc) \) for some \(\alpha >0\). The logarithmic term would generate a discontinuity at the centre, so \(c_1=0\). By rotational symmetry, the trace of \(u\) on \(\partial\disc=\{r=1\}\) is a constant. To determine this constant, we prescribe the additional condition \( \Delta J = 0\) in \(\disc^\mathsf{c}\), which in turn allows us to determine \(c_2\) so that the boundary trace is \(0\), i.e. \(c_2=-\sfrac1 4\), and we obtain \cref{eq:Conv_Disc}.

Analogous formulae for indicator functions of square-shaped regions and their intersections are discussed in \cite[\S19.1]{AMT_Th_2024}.

The rest of this note will be focused on proving Theorem~\ref{th:Main} and providing asymptotic and numerical insights about \cref{eq:Fun_M}. In \cref{sec:Nested}, we derive the first branch of \(E_\varepsilon\) in the regime where the intersection \(\raisebox{.05em}{\smalldisc}\) is itself a disc. \cref{sec:Overlap} establishes a universal bound on \(E_\varepsilon\) and evaluates the integral exactly when \(\raisebox{.05em}{\smalldisc}\) corresponds to a partial overlap of two discs. 
In \cref{sec:Overlap_Disc_inside}, particular attention is provided for the range \( \|x\| \in (1-\varepsilon, 1]\), where two exact formulae are derived for \(F_\varepsilon\), particularly the auxiliary function \(G\) is derived as the indefinite integral of a composite function. We use these workings in \cref{sec:Overlap_Disc_outside} to treat the additional branch \( \| x \| \geq 1\). Finally, \cref{sec:Asymptotics} discusses the asymptotic expansions of \(F_\varepsilon\) near the interface \( \|x\| \approx 1\) as \(\varepsilon\) becomes small and provides numerical insights on its computation under standard and extended precision. All of our numerical experiments are provided in the following repository:
\begin{center}
	\href{https://github.com/DDFT-Modelling/DiscConv}{\texttt{https://github.com/DDFT-Modelling/DiscConv}}
\end{center}

% ––––––––––––––––––––––––––––––––––––––––––––––––––––––––––––––––––––––––––––––––– %
\section{Nested intersection}\label{sec:Nested}

Whenever \( x \in B[0; 1-\varepsilon] \), it holds that \( x - \raisebox{.05em}{\smalldisc} = x - B[x; \varepsilon] = B[0; \varepsilon] \). As a result, the convolution \( E_\varepsilon\) does not depend on the value of \(x\); i.e., it is constant. At this point, let us recall the following important indefinite integral:
\begin{equation}\label{eq:Int_rlogr}
	\int r \log r \dif r = \frac{1}{4} r^2 (\log r^2 - 1) + C.
\end{equation}
The integral follows directly from integration by parts. Here, and from now on, \(C\) represents an arbitrary constant. Now, a pass through polar coordinates and a quick use of \cref{eq:Int_rlogr} reveals
\begin{equation}\label{eq:Constant_Conv}
	E_\varepsilon(x) =
	\frac{1}{2\pi} \int\limits_{ x - \smalldisc } \log |y| \dif y 
	= 
	\frac{1}{2\pi} \int\limits_{0}^{2\pi} \int\limits_{0}^\varepsilon r \log r \dif r \dif \theta
	=
	\frac{1}{4} \varepsilon^2 (\log \varepsilon^2 - 1).
\end{equation}

Observe that when \(\varepsilon \leq \sqrt{\sfrac 1 2} \), then most of the surface defined by the scalar field \( E_\varepsilon\) is constant in \(\disc\). In fact, we will see in the next section that the full convolution is bounded by the absolute value of \cref{eq:Constant_Conv}.

% ––––––––––––––––––––––––––––––––––––––––––––––––––––––––––––––––––––––––––––––––– %
\section{Partial overlap}\label{sec:Overlap}

The case \( \|x\| \in (1-\varepsilon, 1+\varepsilon] \) determines a two-disc intersection with two extreme points since \( x - \raisebox{.05em}{\smalldisc} = B[0;\varepsilon] \cap B[x;1] \). The resulting area resembles a centred ball \(B[0;\varepsilon]$ with a flattened side. As a result of this observation and the fact that \(K\) maintains a constant sign inside \(\disc\), we obtain the following result:

\begin{lemma}\label{lem:bound}
Let \( \varepsilon < 1\). For any \( x \in B[0;1+\varepsilon]\), it holds that
	\begin{equation}\label{lem:bound_eq}
		\big| E_\varepsilon (x) \big| \leq \frac{1}{4} \varepsilon^2 (1 - \log \varepsilon^2).
	\end{equation}
\end{lemma}
\begin{proof}
	Notice that the result is trivial for the nested case since \cref{lem:bound_eq} is nothing else than the absolute value of \cref{eq:Constant_Conv}. 
	For the partial overlap, we have that \( x - \raisebox{.05em}{\smalldisc} \subsetneq B[0;\varepsilon]\), then \(K\) is negative for any \( y \in x - \raisebox{.05em}{\smalldisc}\). Thus, the result follows from
	\[
		\int\limits_{x - \smalldisc} |K|(y)  \dif y
		< \int\limits_{\mathclap{{B_2[0;\varepsilon]}}} -K(u) \dif u,
		\qedhere
	\]
	%and the result follows.
\end{proof}

% ––––––––––––––––––––––––––––––––––––– %
% ––––––––––––––––––––––––––––––––––––– %
\subsection{Disc points}\label{sec:Overlap_Disc_inside}

\begin{wrapfigure}[7]{r}{0.25\textwidth} %this figure will be at the right
    \centering
    \vspace{-2.25em}
    \captionsetup[subfloat]{format=hang,singlelinecheck=false}
    \includegraphics[scale=0.33]{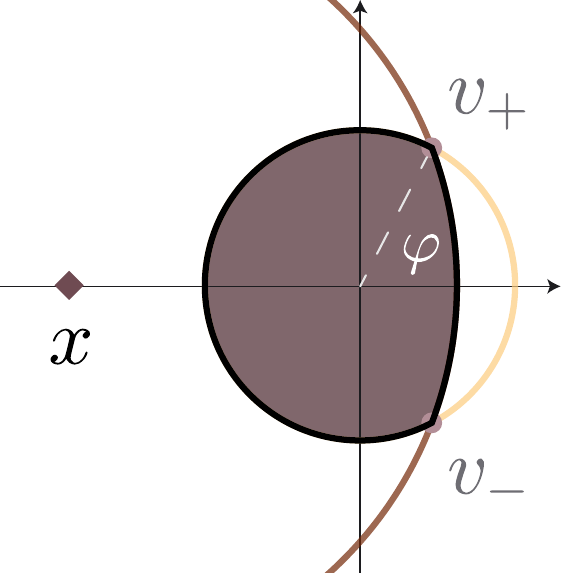}
    \caption{Description of $\circleddash$.}
    \label{fig:RegionRepresentation}
\end{wrapfigure}

In this subsection, we will focus on evaluating \( E_\varepsilon \) for points \( x \in \disc\). Observe that the form of \( x - \raisebox{.05em}{\smalldisc} \) and the radial nature of \(K\) reveal that \( E_\varepsilon \) is radially symmetric. Thus, let us define \( a \coloneqq \|x\| \) and assume without loss of generality that \( x = (-a,0) \).

As anticipated, the region \( \circleddash \coloneqq x - \raisebox{.05em}{\smalldisc} \) is characterised by its extreme points determined by the intersection between two circles, one centred at the origin with radius \(\varepsilon\), and the other centred at \(x\) with unit radius. Letting \( u \coloneqq a^{-1} x \) be the unit vector of \(x\) and \(u_\perp \coloneqq (u_2, -u_1) \) a perpendicular vector to \(u\), we define the quantities \( t \coloneqq \sfrac{ (a^2 + \varepsilon - 1) }{(2a)} \) and \( h \coloneqq \sqrt{\varepsilon^2 - t^2}\). Then the points of intersection \( v_+\) and \(v_-\) can be described by the set
\begin{equation}\label{eq:v_pm}
	v_{\pm} = \cbr{ t u \pm h u_\perp }
	= \cbr{
	\frac{1}{2a}
	\begin{pmatrix}
		1-a^2-\varepsilon^2 ,
		&
		\pm \sqrt{ \big((a+\varepsilon)^2-1\big) \big(1-(a-\varepsilon)^2\big) }
	\end{pmatrix} 
	}.
\end{equation}
Observe that \(v_+\) has a positive vertical component. We denote \(\varphi\) as the angle of \(v_+\) which is given by the expression
\begin{equation}\label{eq:angle_phi}
    \varphi = \arccos -\frac{t}{\varepsilon}  = \operatorname{arccos} { \frac{1 - a^2 - \varepsilon^2}{2a\varepsilon} }.
\end{equation}
A representation of this construction is depicted in \cref{fig:RegionRepresentation}. Here, we observe the effect of the translation by \(x\) of the intersection \( \raisebox{.05em}{\smalldisc}\) and the presence of two extreme points, or \emph{vertices} described in \cref{eq:v_pm}. The angle \cref{eq:angle_phi} is also highlighted.

The previous construction allows us to reduce the integral over \(\circleddash\) to twice the integral over the region
\(
	\rotatebox[origin=c]{-90}{\LEFTcircle} \coloneqq \circleddash \cap \R \times \R_{\geq 0},
\)
which has the following polar description:
\begin{equation}
	\rotatebox[origin=c]{-90}{\LEFTcircle}
	\coloneqq
	\big\{ (r,\theta): \, r \in [0,\varepsilon] \wedge \theta \in [\varphi,\pi] \big\} 
	\cup 
	\big\{ (r,\theta): \, r \in \cbr{0, s(\theta) } \wedge \theta \in [0,\varphi] \big\}.
\end{equation}
The radius $s(\theta) \coloneqq -a \cos(\theta) + \sqrt{1 - a^2 \sin^2(\theta)}$ comes from the polar representation of the constraint $\|v - x\|^2 = r^2 - 2 a r \cos(\theta - \psi) + a^2 = 1$ which determines the maximum radius in the flattened area, and $\psi$ is the angle of $x$, which in this case corresponds to $\psi = \pi$.

As a result, we can write 
\begin{align}
	\int\limits_{x - \smalldisc} K(y)  \dif y 
	&= 2 \int\limits_{ \rotatebox[origin=c]{-90}{\LEFTcircle} } K(y)  \dif y 
	= 
	2 \int\limits_\varphi^\pi \int\limits_0^\varepsilon r k(r) \dif r \dif \theta
	+
	2 \int\limits_0^\varphi \int\limits_0^{s(\theta)} r k(r) \dif r \dif \theta,
	\notag
	\\
	&=
	\frac{1}{4\pi} (\pi - \varphi) \varepsilon^2 ( \log \varepsilon^2 - 1) + 2 F_\varepsilon (a,\varphi),
	\label{eq:Polar_Int_a}
\end{align}
where the first term in \cref{eq:Polar_Int_a} follows from \cref{eq:Int_rlogr}, whereas\footnote{Recall that we are working on the branch \( a \leq 1\).} \( F_\varepsilon (a,\varphi) = \int\limits_0^\varphi \int\limits_0^{s(\theta)} r k(r) \dif r \dif \theta \). We will now focus our attention on determining an exact formula for \(F_\varepsilon\).

% ––––––––––––––––––––––––––––––––––––– %
% ––––––––––––––––––––––––––––––––––––– %
\subsubsection{\(F_\varepsilon\) by a change in the order of integration}\label{ssec:A_ChangeOrder}

In the next subsection, we will prove the following statement:
\begin{lemma}\label{lem:H_a_phi}
\begin{subequations}\label{seq:F_A}
	The integral \(F_\varepsilon (a,\varphi)\) admits a closed-form representation
	\begin{equation}\label{eq:F_by_h_1}
		F_\varepsilon (a,\varphi) = \frac{\varphi}{8\pi} s^2(\varphi) \big( \log s^2(\varphi) - 1\big) - h_1(a;r) \Big|_{r=1-a}^{r=s(\varphi)} \,.
	\end{equation}
	Here, the auxiliary function \(h_1\) is defined in terms of
	\begin{align}
		h_1(a; r) 
		&= \frac{1}{8\pi} \sbr{ r^2 (\log r^2 -1) \operatorname{arccos} \ell(a;r)  +  h_2\big(a; \chi(a;r) \big) },
		\label{eq:h_1_in_h_2}
		\\
		\ell(a;r) &= \dfrac{1- a^2 - r^2}{2 a r},
		\qquad
		\chi(a;r) = \arccos \dfrac{1 + a^2 - r^2}{2 a},
		\label{eq:Handles_for_F_by_h_1}
		\\
		h_2(a; \phi) 
		&= 
		2 \Im \big( \Li ( a e^{i \phi} ) \big)
			+ (1 - a^2)\del{ \phi + \arg\del{ 1 - a e^{i \phi}} }
			 +2a \del{ \log |1-ae^{i \phi}| - 1} \sin \phi.
		\label{eq:h_2_compact}
	\end{align}
\end{subequations}
\end{lemma}

We present the proof of Lemma \ref{lem:H_a_phi} by splitting the result into three parts. Numerical tests are presented in the accompanying repository\footnote{Relevant notebook in code repository: \texttt{\color{xkcdOceanBlue}2 - Convolution inside disc.ipynb}. Here, we
%\begin{itemize}[leftmargin=1.2em, itemsep=0.0em]
%	\item numerically check formula \cref{eq:Constant_Conv},
%	\item symbolically derive the relation between the dilogarithm and the log-cosine integrand using SymPy,
%	\item numerically check \cref{eq:h_2_a} and \cref{eq:h_2_b} against a quadrature implementation,
%	\item implement \cref{eq:F_by_h_1} and check it against three different quadrature methods.
%\end{itemize}
\textbf{(a)} numerically check formula \cref{eq:Constant_Conv};
\textbf{(b)} symbolically derive the relation between the dilogarithm and the log-cosine integrand using SymPy;
\textbf{(c)} numerically check \cref{eq:h_2_a} and \cref{eq:h_2_b} against a quadrature implementation; and
\textbf{(d)} implement \cref{eq:F_by_h_1} and check it against three different quadrature methods.
}.

\paragraph{Integration order:}
First, let us derive \cref{eq:F_by_h_1} by an adequate change in the order of integration. Recall \(F_\varepsilon\) from \cref{eq:Polar_Int_a}, then, for a fixed \(r>0\), the condition \(r \leq s\) is equivalent to \( r + a \cos \theta \leq \sqrt{ 1 - a^2 (1 - \cos^2\theta) }\). Squaring both sides, we obtain that \( r^2 + 2 a r \cos (\theta) \leq 1-a^2  \) and %\( \cos \theta \leq \sfrac{(1-a^2-r^2)}{(2ar)} \eqqcolon \ell(r) \).
\( \cos \theta \leq \ell(a;r) \).

Since \(a>0\) and \(r \geq 0\), we have that \(\ell(a;r) \geq -\sfrac{r}{a}\), thus for each \(r\), the admissible \(\theta\) lie within \( \cos(\theta) \leq \ell(r) \). As the cosine function decreases on \([0,\pi]\) and \(\varphi < \pi\), we obtain the admissible interval
\( \sbr{ T(r), \varphi}\) for \(\theta\), where
\[
    T(r) = 
    \begin{cases}
        0 & \text{if } \ell(r) > 1,
        \\
        \arccos \ell(a;r), & \text{if } \cos(\varphi) \leq \ell(a;r) \leq 1.
    \end{cases}
\]
The threshold $\ell(a;r) > 1$ is equivalent to $r < 1-a$. In contrast, $r$ varies from $0$ to the top of the region determined by $s(\varphi)$, as the function $s$ is monotone increasing in \([0,\pi]\).
We thus obtain
\begin{align}
	F_\varepsilon (a,\varphi) 
	&= \int\limits_0^{s(\varphi)} \int\limits_{T(r)}^\varphi  r k(r) \dif \theta  \dif r 
	= \int\limits_0^{s(\varphi)}  \del{ \varphi - T(r)}  r k(r) \dif \theta  \dif r 
	= \int\limits_0^{s(\varphi)} \varphi r k(r) \dif r - h_1 (a;r) \Big|_{r=1-a}^{r=s(\varphi)} \,.
	\label{eq:H_a_phi}
\end{align}
The first term at the right-hand side of \cref{eq:H_a_phi} follows once again from \cref{eq:Int_rlogr}. For the second term, we define
\begin{equation}\label{eq:h_1}
	h_1 (a;r) \coloneqq \frac{1}{2\pi}
	%\int \arccos\left( \frac{1-a^2-r^2}{2ar} \right) r \log r \dif r,
	\int \arccos\left( \ell(a;r) \right) r \log r \dif r,
\end{equation}
and thus we have shown the characterisation \cref{eq:F_by_h_1}.

% ––––––––––––––––––––––––––––––––––––– %
% ––––––––––––––––––––––––––––––––––––– %
\paragraph{First primitive:} 

Now we focus on deriving an expression for the exact indefinite integral \cref{eq:h_1}. This is possible by an adequate transformation which will yield \cref{eq:h_1_in_h_2}.
In this context, let us define the change of variables \( u \equiv \arccos \ell(a;r) \) followed by \(\dif v \equiv r \log r\); these yield
\[
    h_1(a;r) = \frac{1}{2\pi} \left[ \frac{1}{4}r^2 (\log r^2 -1) \operatorname{arccos} \ell(a;r) 
    + \frac{1}{4} \int  r^2 (\log r^2 -1) \frac{ \partial_r \ell (a;r) }{\sqrt{1 - \ell^2(a;r)}} \dif r  
    \right].
\]
Here, we have that 
%\( \partial_r \ell(a;r) = \sfrac{(a^2 - r^2 - 1)}{(2ar^2)} \) and 
%\( 
%	\sqrt{1 - \ell^2(a;r)} = (\sfrac{1}{2ar}) \sqrt{ ((a+1)^2 - r^2 ) (r^2 - (1-a)^2 ) }
%	%1 - \ell^2(a;r) = (\sfrac{1}{2ar})^2 \big( \del{(a+1)^2 - r^2} \del{r^2 - (1-a)^2} \big)
%\).
\[
	\partial_r \ell(a;r) = \frac{(a^2 - r^2 - 1)}{(2ar^2)} 
	\qquad\text{and}\qquad
	\sqrt{1 - \ell^2(a;r)} = \frac{1}{2ar} \sqrt{ \del{ (a+1)^2 - r^2 } \del{ r^2 - (1-a)^2 } }.
\]
Introducing the aditional change of variables \(\phi \equiv \arccos \sfrac{(1 + a^2 - r^2)}{(2a)}\)  which maps the range $[1-a,1+a]$ to $[0,\pi]$ through the relationship $r^2 \equiv 1 + a^2 - 2a \cos \phi$, we obtain, through implicit differentiation, that $ r \dif r \equiv a \sin(\phi) \dif \phi$. Moreover,
\(
    ((a+1)^2 - r^2) (r^2 - (1-a)^2) 
    %= \big( 2 a - (r^2 - 1 - a^2) \big) \big( (r^2 - 1 - a^2) + 2 a   \big)
    \equiv 
    %4a^2 (1 + \cos(\phi))(1 - \cos(\phi) ) = 
    4 a^2  \sin^2(\phi)
\)
and the nonnegativity of sine over \([0,\pi]\) yields that \( 2ar \sqrt{1 - \ell^2(r)} \equiv 2 a \sin \phi\). Thence:
%\begin{align*}
%    r^2 \frac{ \ell'(r) }{\sqrt{1 - \ell^2(r)}} \dif r  \equiv (a \cos \phi - 1) \dif \phi
%\end{align*}
%and we can write
\begin{align}
    h_2(a;\phi) &\coloneqq \int  r^2 (\log r^2 -1) \frac{ \partial_r \ell(a;r) }{\sqrt{1 - \ell^2(a;r)}} \dif r   \bigg|_{ r^2 = 1 + a^2 - 2a \cos \phi}
    \notag
    \\
    &=
    -\int (1 - a \cos \phi) \big( \log(1 + a^2 - 2a \cos\phi ) -1  \big) \dif \phi.
    \label{eq:h_2}
\end{align}
Here we have unveiled \cref{eq:h_1_in_h_2} and \cref{eq:Handles_for_F_by_h_1}, where \(\chi\) results from evaluating the change of variables that gave us \(\phi\).

% ––––––––––––––––––––––––––––––––––––– %
% ––––––––––––––––––––––––––––––––––––– %
\paragraph{Second primitive:} 

Let us expand and exactly integrate each term in \cref{eq:h_2}. First, notice that \(\log(1+ a^2 - 2a \cos \phi) = \log |1 - a e^{i\phi}|^2 = \log(1-a e^{i\phi}) + \log(1-a e^{-i\phi})\). Now, the dilogarithm satisfies the equation \( \frac{\dif}{\dif z} \Li (z) = \sfrac{-\log(1-z)}{z}\) for any complex \(z\in \mathbb{C}\). Then, the chain rule allows us to write
\[
    \frac{\partial }{\partial \phi} \Li ( a e^{i \phi} ) 
    = \Li '(a e^{i \phi} ) \frac{\partial }{\partial \phi} (a e^{i \phi} )
    = -\frac{\log(1-a e^{i \phi})}{a e^{i \phi}} a i e^{i \phi} = -i\log(1-a e^{i \phi}).
\]
As a result, we have the primitive
\(
	\int \log(1-a e^{\pm i  \phi}) \dif \phi = \pm i \, \Li ( a e^{i \phi} ) + C
\),
and thus
\begin{equation}\label{eq:h_2_a}
	\boxed{ \int \log(1 + a^2 - 2a \cos\phi )  \dif \phi 
	= - 2 \Im \big( \Li ( a e^{i \phi} ) \big) + C.}
\end{equation}
Notice that if \(\phi = 0\), then \cref{eq:h_2_a} equals \(C\).

To compute the indefinite integral of \( (a \cos \phi) \log(1+ a^2 - 2a \cos\phi) \) can be done by means of the relation \(  \cos \phi = (\sfrac{1}{2}) (e^{i\phi} + e^{-i\phi}) \). Using the change of variable \( z \equiv a e^{i\phi} \), and the fact that \( \Im \del{ \log(1-a e^{i\phi}) } = -\operatorname{arctan} \frac{a \sin\phi}{1-a \cos\phi} \), we obtain
\begin{equation}\label{eq:h_2_b}
\boxed{
	\begin{aligned}
		\int (a \cos\phi) \log(1 + a^2 - 2a \cos\phi )  \dif \phi 
		&= 
		-a^2 \phi + (1-a^2)\arctan\left(\frac{a\sin\phi}{1-a\cos\phi}\right) 
		\\& 
		\qquad + a\big( \log(1-2a\cos\phi+a^2) - 1 \big)\sin\phi  + C.
	\end{aligned}
	}
\end{equation}
Whenever $a$ or $\phi$ are $0$, when the latter coincides with \( \chi(a,1-a)\), then \cref{eq:h_2_b} equals $C$. Similarly, if $a=1$, the integral reduces to $- \phi + \big( \log(2 - 2 \cos\phi) - 1 \big) \sin \phi$ which is continuous at $0$. 

Finally, we have that $\int a \cos \phi -1 \dif \phi = -\phi + a \sin \phi + C$. Thus, we can write $h_2$ in the form 
\begin{equation}
	\begin{aligned}
		h_2(a;\phi) &= 
		2 \Im \big( \Li ( a e^{i \phi} ) \big)
		+ (1-a^2)\left[ \phi + \arctan\left(\frac{a\sin\phi}{1-a\cos\phi}\right)\right] 
		\\
		&\qquad + a\big( \log(1-2a\cos\phi+a^2) - 2 \big)\sin\phi + C.
	\end{aligned}
\end{equation}
which is equivalent to \cref{eq:h_2_compact} where we have set the integration constant to zero.

% ––––––––––––––––––––––––––––––––––––– %
% ––––––––––––––––––––––––––––––––––––– %
\subsubsection{\(F_\varepsilon\) by direct integration}\label{ssec:B_Direct}

In \cref{ssec:A_ChangeOrder} we reversed the order of integration, obtaining a constructive representation of \(F_\varepsilon\) in \cref{eq:Polar_Int_a} using Lemma \ref{lem:H_a_phi}. An issue with \cref{seq:F_A} is that its construction inherently relies on the assumption that \( a \leq 1\) and thus any expression for the complementary region \( a \in (1,1+\varepsilon)\) would have to be rederived based on the cutting points of \(s\) given by \(\alpha\). A further limitation becomes evident from the behaviour of \( F_\varepsilon\) for small \(\varepsilon\): the function decreases almost quadratically in \(\varepsilon\). As a result, any numerical implementation of \cref{seq:F_A} in finite precision is susceptible to cancellation and round-off errors unless special care is taken. Indeed, preliminary numerical tests showed that these issues appear quickly\footnote{Relevant notebook in code repository: \texttt{\color{xkcdOceanBlue}3 - Direct and Error.ipynb}. Here, we
\textbf{(a)} symbolically verify \cref{eq:G_alt_a} and \cref{eq:G_alt_b};
\textbf{(b)} implement \cref{eq:H_Full} and check it against \cref{eq:F_by_h_1}, and different quadrature methods;
\textbf{(c)} visualise scaling and numerical errors.
}.

For these reasons, it is useful to seek an alternative expression for the integrand defining \( F_\varepsilon\). Particularly, we aim to find one representation that extends naturally to the entire domain and provides an asymptotic expansion suitable for highly accurate numerical evaluation. Developing such an expression is the goal of this subsection. For this task, let us notice that \cref{eq:Int_rlogr} once more allows us to obtain
\begin{equation}\label{eq:alternative_H}
	F_\varepsilon(a;\varphi) = \int\limits_0^\varphi \frac{1}{8\pi} s^2(\theta) \del{ 2 \log {s(\theta)} - 1 } \dif \theta.
\end{equation}

\begin{lemma}\label{lem:H_a_phi_2}
	The integral \(F_\varepsilon (a,\varphi)\) admits an alternative closed form 
	\begin{equation}\label{eq:H_Full}
	\begin{aligned}
		F_\varepsilon (a,\varphi) &= \frac{1}{4 \pi} \bigg[
		 %\Im \big( \Li (-a e^{2i\phi}) \big)
		 \Im  \Li (z)
		 + (1-a^2) \del{ \phi - \frac{1}{2}  %\arctan\left(\frac{ a \sin 2\phi}{1 + a \cos 2\phi} \right)  \right]
		 %\arg(1 + a e^{2i\phi}) }
		 \arg(1 - z) }
		%\\
		%&\qquad
		%+ \frac{a}{2} \big( 2- \log(1+ a^2 +2a\cos 2\phi)  \big)\sin 2\phi \bigg] 
		%+ a \big( 1- \log|1 + ae^{2i\phi}|  \big)\sin 2\phi \bigg] 
		+ a \big( 1- \log|1 - z|  \big)\sin 2\phi \bigg] 
		\Bigg|_{\phi = \Phi(\varphi) } \hspace{-0.5em} - \frac{1}{8} (1-a^2),
	\end{aligned}
	\end{equation}
	where \( z \coloneqq -a e^{2i\phi} \) and \( \Phi(\theta) = \frac{1}{2} \arccos \del{ \frac{s^2(\theta) - 1- a^2 }{2a} } \).
\end{lemma}
\begin{proof}
We will express \cref{eq:H_Full} as the difference of two evaluations of the following indefinite integral under a suitable change of variables:
\begin{equation}\label{eq:alternative_H_indefinite}
	M(a;\theta) \coloneqq \int s^2(\theta) \del{ 2 \log {s(\theta)} - 1 } \dif \theta.
\end{equation}
If we introduce the change of variables \( u \equiv s(\theta) \) which is guaranteed by the monotonicity of \(s\), we obtain that
\begin{align*}
	\frac{\dif u}{\dif \theta} &
	%= a \sin \theta - \frac{a^2 \sin\theta \cos\theta}{\sqrt{ 1 - a^2 \sin^2 \theta }}
	= \frac{a u \sin\theta}{ u + a \cos\theta},
	\qquad
	%\\
	\cos \theta = \frac{ 1-a^2 - u^2 }{2au},
	\qquad\text{and}\qquad
	\sqrt{ 1- a^2 \sin^2 \theta } = \frac{ 1 + u^2 - a^2 }{2u}.
\end{align*}
%If we square \(s\), we can further arrive at the relation \( \cos \theta = \sfrac{ (1-a^2 - u^2) }{(2au)}\) and \( \sqrt{ 1- a^2 \sin^2 \theta } = \sfrac{ (1 + u^2 - a^2) }{(2u)}\). 
Given that \(\varphi \in (0,\pi]\), we can assert that \( \sin\theta = \sqrt{1 - \cos^2\theta} \) and thus
\[
	\dif \theta \equiv \frac{u+a\cos \theta}{a u \sin\theta} = \frac{1 + u^2 - a^2}{u \sqrt{ ( (1+a)^2 - u^2 ) (u^2 - (1-a)^2) } } \dif u.
\]
Employing the additional change of variables \( t \equiv u^2\), allows us to write \cref{eq:alternative_H_indefinite} as
\begin{equation}\label{eq:alternative_H_indefinite_b}
	M(a;\theta)
	= \frac{1}{2}
	\int \del{\log t - 1 } \frac{1 + t - a^2}{ \sqrt{ ( (1+a)^2 - t ) (t - (1-a)^2) } } \dif t
	\Bigg|_{t = s^2(\theta) } 
	.
\end{equation}
Observe the natural range \( t \in \sbr{ (1-a)^2, (1+a)^2 }\).

Now let us introduce the last change of variables \( t \equiv (1+a)^2 \cos^2\phi + (1-a)^2 \sin^2\phi =  1 + a^2 + 2a \cos 2\phi \), which yields \( \dif t \equiv -4a \sin 2\phi \dif \phi\). Similarly, \( (1+a)^2 - t \equiv 4a \sin^2 \phi\), \( t - (1-a)^2 \equiv 4a \cos^2 \phi\), and 
\( t+1-a^2 \equiv 2(1+a \cos 2\phi ) \). As a result, we can write \cref{eq:alternative_H_indefinite_b} as the composite function
\begin{equation}\label{eq:alternative_H_indefinite_c}
	G(a;\phi)
	\coloneqq
	%\frac{1}{2} \int \del{\log ( 1 +  a^2 + 2a \cos 2\phi ) - 1 } \frac{ 2(1+a \cos 2\phi ) }{ 2 a \sin 2 \phi } (-4a \sin 2 \phi) \dif \phi.
	-2 \int (1+a \cos 2\phi ) \del{\log ( 1 +  a^2 + 2a \cos 2\phi ) - 1 } \dif \phi,
\end{equation}
where \( G(a;\phi) = M(a;\theta) \) under the relation \( \phi = \Phi(\theta) \coloneqq \frac{1}{2} \arccos L(\theta)\), %\del{ \frac{s^2(\theta) - 1- a^2 }{2a} }\).
where \( L(\theta) \coloneqq \frac{s^2(\theta) - 1- a^2 }{2a} \).

At this point, notice the resemblance to \cref{eq:h_2}. Indeed, it is possible to compute \cref{eq:alternative_H_indefinite_c} using a similar procedure. First, observing that \( \log (1+ a^2 +2 a \cos 2\phi) = \log |1 + a e^{2i\phi}|^2 = \log(1 + ae^{2i\phi}) + \log(1 + ae^{-2i\phi}) \) and following the procedure we used to obtain \cref{eq:h_2_a}, we arrive at
\begin{equation}\label{eq:G_alt_a}
	\boxed{
	\int \log (1+ a^2 +2 a \cos 2\phi) \dif \phi = -\Im \big( \Li ( -a e^{2i \phi} ) \big) + C.
	}
\end{equation}
Similarly, using the complex representation of cosine, we obtain
\begin{equation}\label{eq:G_alt_b}
	\boxed{
	\begin{aligned}
	\int (a \cos 2\phi) \log (1+ a^2 +2 a \cos 2\phi) \dif \phi 
	&= 
		%\frac{1}{2} \bigg[ (1+a^2) \phi + (1-a^2)\arctan\left(\frac{ a-1}{a+1} \tan \phi\right) 
		%\\&\quad
		%+ a\big( \log(1+ a^2 +2a\cos 2\phi) - 1 \big)\sin 2\phi \bigg] + C.
		\frac{1}{2} \bigg[ 2 a^2 \phi + (1-a^2)\arctan\left(\frac{ a \sin 2\phi}{1 + a \cos 2\phi} \right) 
		\\&\quad
		+ a\big( \log(1+ a^2 +2a\cos 2\phi) - 1 \big)\sin 2\phi \bigg] + C.
	\end{aligned}
	}
\end{equation}
Finally, we have that \( \int (1+a \cos 2\phi) \dif \phi = \phi + \frac{a}{2} \sin 2\phi + C\). 

Collecting terms, we arrive at
\begin{equation}\label{eq:G_a_phi_theta}
\begin{aligned}
	G\big(a;\phi \big) %\Big|_{\phi = \Phi(\theta)} 
	&=
%	2 \Im \big( \Li ( -a e^{2i \phi} ) \big) + 2\phi + a \sin 2\phi 
%	- 
%	\bigg[ 2 a^2 \phi + (1-a^2)\arctan\left(\frac{ a \sin 2\phi}{1 + a \cos 2\phi} \right) 
%		\\&\quad
%		+ a\big( \log(1+ a^2 +2a\cos 2\phi) - 1 \big)\sin 2\phi \bigg] + C.
	2 \Im \big( \Li (-a e^{2i\phi}) \big)
	+ (1-a^2) \left[ 2 \phi - \arctan\left(\frac{ a \sin 2\phi}{1 + a \cos 2\phi} \right)  \right]
	\\
	&\qquad
	+ a\big( 2- \log(1+ a^2 +2a\cos 2\phi)  \big)\sin 2\phi + C.
\end{aligned}
\end{equation}

To compute the iterated integral \cref{eq:alternative_H} as the difference between two evaluations of \(G\), the following transformations must take place:
\[
	\Big|_{\theta = 0}^{\theta = \varphi}
	\to 
	\Big|_{u = s(0)}^{u = s(\varphi)} 
	\to
	\Big|_{t = s^2(0) }^{t = s^2(\varphi)} 
	\to 
	%\Big|_{\phi = \frac{1}{2} \arccos \del{ \frac{s^2(0) - 1- a^2 }{2a} } }^{ \phi = \frac{1}{2} \arccos \del{ \frac{s^2(\varphi) - 1- a^2 }{2a} } } 
	\Big|_{\phi = \Phi(0) }^{ \phi = \Phi(\varphi) } 
\]
Replacing each value accordingly, we obtain the following two limits \( \Phi(0) = \sfrac{\pi}{2}\) and \( \phi(\varphi) = \frac{1}{2} \arccos L(\varphi) \) where 
\[
L(\varphi) =\begin{cases}
            \dfrac{ \varepsilon^2  - 1 - a^2}{ 2 a } & \text{if } a^2 \leq 1 + \varepsilon^2,
            \\
            \dfrac{ (a^2 - 1)^2 }{2 a \varepsilon^2} - \dfrac{1+a^2}{2a} & \text{if } a^2 > 1 + \varepsilon^2,
    \end{cases}
\]
and notice the special values \( L(\varphi) \big|_{a=1\pm\varepsilon} = \pm 1\). 
%This gives the special values \( G(1+\varepsilon;  0) = \pi \varepsilon(1-\varepsilon) + C \)
Moreover, at \(\sfrac{\pi}{2}\), we have 
%a closed-form value for \(G\) by direct evaluation of \cref{eq:G_a_phi_theta}:
by direct evaluation of \cref{eq:G_a_phi_theta} that
\(
	G(a; \Phi(0)) = (1-a^2) \pi + C.
\)
Thus, we obtain that \cref{eq:alternative_H} is just
\begin{equation}\label{eq:Simple_F_first_branch_G_and_Quad}
	F_\varepsilon(a;\varphi)  = \frac{1}{8\pi} \sbr{ G(a; \Phi(\varphi)) - G(a; \sfrac \pi 2) }
	= \frac{1}{8\pi} G(a; \Phi(\varphi)) - \frac{1}{8} (1-a^2),
\end{equation}
which coincides with \cref{eq:H_Full}.
\end{proof}

Observe that \cref{eq:G_a_phi_theta} is well defined for any \(a\). We will use \(G\) to further characterise the convolution values outside of \(\disc\).

\begin{remark}
	We can highlight the following reduction from \cref{eq:Intersection_Angle_Cos}:
\[
	\log\del{1+ a^2 +2a L(\varphi) } = 
	\begin{cases}
            2 \log \varepsilon & \text{if } a^2 \leq 1 + \varepsilon^2,
            \\
            2 \log \frac{a^2-1}{\varepsilon} & \text{if } a^2 > 1 + \varepsilon^2.
    \end{cases}
\]
\end{remark}

% ––––––––––––––––––––––––––––––––––––– %
% ––––––––––––––––––––––––––––––––––––– %
\subsection{Outer points}\label{sec:Overlap_Disc_outside}

In this subsection, we will focus on evaluating \( E_\varepsilon \) for points \( x \in \disc^{\mathsf c}\). Using the same construction as in \cref{sec:Overlap_Disc_inside}, by symmetry it suffices to consider points of the form \( x =(-a,0)\) where \( a \in [1,1+\varepsilon]\). For any points of larger magnitude, observe that \cref{eq:Newt_Pot} must be zero.

A key difference from the previous case lies in the description of the set
\(
	\rotatebox[origin=c]{-90}{\LEFTcircle} = \circleddash \cap \R \times \R_{\geq 0},
\), where recall \( \circleddash \coloneqq x - \raisebox{.05em}{\smalldisc} \). The curve \( s(\theta) = -a \cos \theta + \sqrt{ 1 - a^2 \sin^2 \theta } \) still characterises the unit circle centred at \(x\), but it is no longer defined continuously on all of \( [0,2\pi]\). Instead, \(s\) displays three disjoint intervals of continuity. Indeed, the square-root term in \(s\) becomes undefined when \( |\sin \theta| >  \sfrac{1}{a}\). Defining \( \alpha \coloneqq \arcsin \sfrac{1}{a} \in [\sfrac{\pi}{6}, \sfrac{\pi}{2}] \), then \(s\) is precisely well defined over the following intervals accompaigned by their corresponding images:
\begin{equation}\label{eq:theta_regions_s}
	\theta \in
	\begin{cases}
		[0,\alpha]
		&
		\text{covers the arc} 
		\quad
		\sbr{(1-a,0), e_- },
		\\
		[\pi - \alpha, \pi+\alpha]
		&
		\text{covers the arc} 
		\quad
		\sbr{e_+, e_- },
		\\
		[2\pi - \alpha, 2\pi] \qquad
		&
		\text{covers the arc} 
		\quad
		\sbr{e_+, (1-a,0) };
	\end{cases}
\end{equation}
and \( e_\pm \coloneqq \del{-a + \sfrac 1 a, \pm \sqrt{1 - \sfrac 1 {a^2} }} \). 

The relationship between the point \(e_{+}\) and the point \(v_{+}\) (see \cref{eq:v_pm}) determines whether \( \rotatebox[origin=c]{-90}{\LEFTcircle} \) is described by one or two of the angular subregions in \cref{eq:theta_regions_s}.  By comparing their horizontal coordinates, we obtain the following angular description of the region cut out by the disc \(B[x;1]\):
\begin{equation}\label{eq:Theta_Angular_Subregions}
	\Theta(a) \coloneqq
	\begin{dcases*}
		[\pi - \alpha, \varphi]
		& if \( a = 1 \),
		\\
		[2\pi -\alpha, 2\pi] \cup [\pi - \alpha, \varphi]
		&
		if \(1 < a^2 < 1 + \varepsilon^2\),
		\\
		[\varphi + \pi, 2\pi] 
		& if \( 1 + \varepsilon^2 \leq a^2 \leq (1+\varepsilon)^2 \).
	\end{dcases*}
	% But the direction of the integral might be the other way around
\end{equation}
Employing \cref{eq:Theta_Angular_Subregions}, the following representation is revealed:
\begin{equation}
	\rotatebox[origin=c]{-90}{\LEFTcircle}
	=
	\big\{ (r,\theta): \, r \in [0,\varepsilon] \wedge \theta \in [\varphi,\pi] \big\} 
	\cup 
	\big\{ (r,\theta): \, r \in \cbr{0, s(\theta) } \wedge \theta \in \Theta(a) \big\}.
\end{equation}
The circular segments corresponding to \(\Theta(a)\) inside \( \rotatebox[origin=c]{-90}{\LEFTcircle} \) are illustrated in \cref{fig:AngularRegions_Out}. In particular, we observe that the point \(e_{+}\) determines whether additional angular intervals are required in the case \( 1 < a^{2} < 1 + \varepsilon^{2} \).

% ------------------------------------- %
\begin{figure}[h!]
\centering
\captionsetup[subfloat]{format=hang}
	\subcaptionbox{\(a=1\)}[3cm]{
	\includegraphics[scale=0.3]{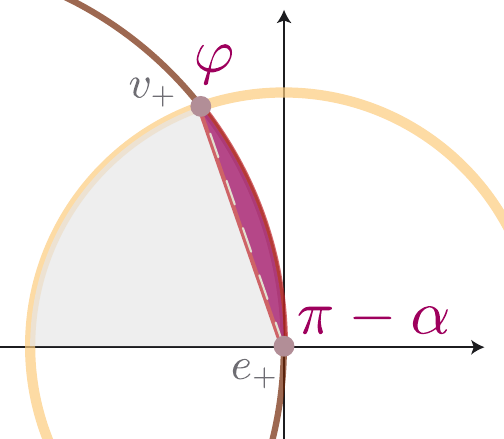}}
	\hspace{1.25cm}
    \subcaptionbox{\( 1 < a^2 < 1+\varepsilon^2 \)}[3cm]{
    \includegraphics[scale=0.3]{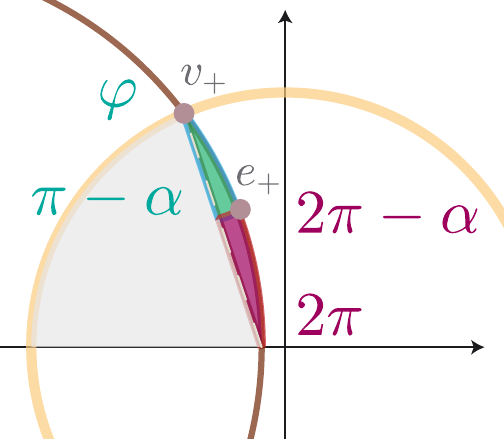}
	}
	\hspace{1.25cm}
    \subcaptionbox{\( 1+\varepsilon^2 \leq a^2 \leq (1+\varepsilon)^2 \)}[3cm]{
    \includegraphics[scale=0.3]{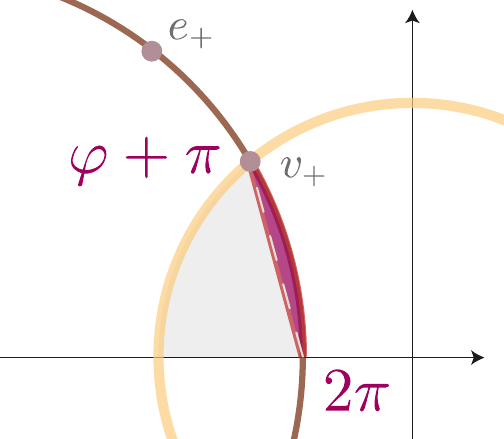}
	}
	%
	%
    %\vspace{1em}
    \caption{Circular segments defined by \(\Theta(a)\).}
\label{fig:AngularRegions_Out}
%\vspace{-1\baselineskip}
\end{figure}
% ------------------------------------- %

Recalling \cref{eq:Polar_Int_a} and \cref{eq:alternative_H}, we obtain
\begin{align*}
	\int\limits_{x - \smalldisc} K(y)  \dif y 
	&= 2 \int\limits_{ \rotatebox[origin=c]{-90}{\LEFTcircle} } K(y)  \dif y 
	= 
	\frac{\pi - \varphi}{4\pi} ( \log \varepsilon^2 - 1) \varepsilon^2
	+
	2 M(a) 
\end{align*}
where we define the branch of \( F_\varepsilon \) over \( [1,1+\varepsilon]\) by
\begin{equation}
\label{eq:alternative_H_general}
	F_\varepsilon (a;\varphi) \coloneqq \frac{1}{8\pi}  \int\limits_{ \Theta(a) } s^2(\theta) \del{ 2 \log {s(\theta)} - 1 } \dif \theta.
\end{equation}
Using \cref{eq:alternative_H_indefinite} and \cref{eq:G_a_phi_theta}, we may express \(F_\varepsilon\) in terms of \(G\) as
\[
	F_\varepsilon (a;\varphi) =
	\frac{1}{8\pi} 
	\begin{dcases*}
		G\del{a; \Phi(\varphi) } - G\big(a; \sfrac{\pi}{2} \big)
		& if \( a = 1 \),
		% Here \alpha = \pi/2 and \pi-\alpha = \alpha
		% This is always negative because \varphi > \alpha
		\\
		G\big(a; \Phi(\varphi) \big) + G\big(a; \sfrac \pi 2 \big) - 2 G\big(a; (\sfrac 1 4) (\pi + 2\alpha) \big)
		% Always negative
		&
		if \(1 < a^2 < 1 + \varepsilon^2\),
		\\
		G\big(a; \sfrac \pi 2 \big) - G\big(a; \Phi(\varphi+\pi) \big)		% Always negative, collapses to 0
		& if \( 1 + \varepsilon^2 \leq a^2 \leq (1+\varepsilon)^2 \),
	\end{dcases*}
\]
where \( \Phi(\theta) = \frac{1}{2} \arccos \del{ \frac{s^2(\theta) - 1- a^2 }{2a} } \), and we have used the relations presented in \cref{tb:Relevant_Quantities_and_Values} (see \cref{App:Tables}). Applying the tabulated values from \cref{tb:Values_of_G}, we simplify \(F_\varepsilon\) to
\begin{equation}
\label{eq:Solution_M}
	F_\varepsilon (a;\varphi)  =
	\frac{1}{8\pi} 
	\begin{dcases*}
		G\del{1; \Phi(\varphi) } 
		& if \( a = 1 \),
		\\
		G\big(a; \Phi(\varphi) \big) 
		%+G\big(a; \sfrac \pi 2 \big)
		%+\del{(1-a^2) \pi - 2\pi \ln a}
		- 2\pi \ln a
		 - 2 G\big(a; (\sfrac 1 4) (\pi + 2\alpha) \big)
		&
		if \(1 < a^2 < 1 + \varepsilon^2\),
		\\
		%G\big(a; \sfrac \pi 2 \big)
		%(1-a^2) \pi 
		- 2\pi \ln a
		 - G\big(a; \Phi(\varphi+\pi) \big)
		& if \( 1 + \varepsilon^2 \leq a^2 < (1+\varepsilon)^2 \),
		\\
		0 & if \( a = 1+\varepsilon\).
	\end{dcases*}
\end{equation}

% ––––––––––––––––––––––––––––––––––––– %
\subsubsection{The branch at \(\phi = \sfrac \pi 2\) and monodromy}\label{sec:Monodromy}

In this section, we explain the value of \(G(a; \sfrac{\pi}{2})\) appearing in \cref{eq:Solution_M}. For \( a \geq 1\), the function 
\(
	f: \operatorname{dom} f \ni \theta  \longmapsto 
	s^2(\theta) \del{ \log s^2 (\theta) -1 } \in \R
\)
is well defined on its natural domain \( \operatorname{dom} f \coloneqq \cbr{ [k\pi \mp \alpha]: k \in \Z }\), where \( \alpha = \arcsin \sfrac 1 a \). %Particularly, \( \alpha \leq \sfrac{\pi}{2}\) and can only attain the equality at \( a = 1\). 
%By continuity of \(M\), this means that we will approach the value \(\phi = \sfrac \pi 2\) from the left for \( a < 1\) and from the right for \( a>1\). 
We can write \cref{eq:G_a_phi_theta} as
\begin{equation}\label{eq:G_a_phi_arg}
\begin{aligned}
	G(a;\phi) &=
	2 \Im \Li (-a e^{2i\phi})
	+ (1-a^2) \left( 2 \phi - \arg\del{1 + a e^{2i\phi}}  \right)
	\\
	&\qquad
	+ a\big( 2- \log\del{1+ a^2 +2a\cos 2\phi}  \big)\sin 2\phi + C.
\end{aligned}
\end{equation}
For the remainder of this section, we fix \(C = 0\). At \( \phi = \sfrac{\pi}{2}\), it holds that \( \arg \del{1 + a e^{2i\phi}} = \pi\) and \(\sin 2\phi = 0\), for which \cref{eq:G_a_phi_arg} reduces to 
\begin{equation}\label{eq:G_a_pi_over_2_reduced}
	G \del{a;\frac{\pi}{2} }
	= 2 \Im \Li \del{-a e^{i\pi} }
	= 2 \Im \Li (a \pm i0).
\end{equation}
The dilogarithm \( \Li(z)\) is a multivalued analytic function with branch points at \(z=0\) and \(z=1\). Analytic continuation around these branch points produces nontrivial monodromy (characterised by the Heisenberg group): looping once around \(z=1\) changes the value of \(\Li(z)\) by specific additive terms (reflecting the discontinuities across the branch cuts) \cite{OSullivan2016a,Zagier2007a}. Thus, it remains to determine which boundary value of \(\Li\) on the cut \((1,\infty)\) is selected in \cref{eq:G_a_pi_over_2_reduced}.

Before identifying the relevant boundary value, note that in our application, \(a\) varies continuously and the turning angle \(\alpha\) satisfies \(\alpha = \sfrac{\pi}{2} + i\delta \) with\footnote{Specifically, \( \arcsin \sfrac{1}{a} = \sfrac{\pi}{2} - i \log \del{ \sfrac{1}{a} + \sqrt{ \sfrac{1}{a^2} - 1 } } \).} \(\Im \delta < 0\) for \(a < 1\), and
\(
  \alpha|_{a=1}  = \sfrac{\pi}{2}
\).
Thus, as \(a\to 1\), the endpoint \(\phi = \sfrac{\pi}{2}\) is reached \emph{from below} in the complex plane.  This continuity constraint determines the analytic branch that must be used in the evaluation of \(G(a;\phi)\).

Now, consider the path \([0,\sfrac \pi 2] \ni \phi \mapsto z(\phi) = -a e^{2i\phi}\). A direct computation shows that \( \Im z(\phi) = -a\sin 2\phi < 0 \) throughout the path. % bc 0 < \phi < \pi/2
Therefore \(z(\phi)\) lies in the lower half-plane and approaches the point \(a>1\) on the real axis from below: 
%as \(\phi \uparrow \sfrac{\pi}{2}\): 
\( \lim\limits_{ \phi \uparrow \sfrac \pi 2 } z(\phi) = a - i 0\).

The analytic continuation of \(\Li\) along this path gives
\(
	\lim\limits_{ \phi \uparrow \sfrac \pi 2 } \Li \del{-a e^{2i\phi} } = \Li(a - i0)
\). Substituting this into \eqref{eq:G_a_pi_over_2_reduced} and using the standard jump relation
\(
	\Im \Li(a \pm i0) = \pm \pi \ln a, %\qquad (\forall a>1),
\), for all \( a > 1\),
we obtain
\begin{equation}\label{eq:G_a_pi_over_2_final}
	G\Bigl(a;\frac{\pi}{2}\Bigr)
	= 2\,\Im \Li(a - i0)
	= -2\pi \ln a,\qquad a>1.
\end{equation}
This confirms the value used in \cref{eq:Solution_M}.

% ––––––––––––––––––––––––––––––––––––––––––––––––––––––––––––––––––––––––––––––––– %
\section{Asymptotics and numerical treatment}\label{sec:Asymptotics}

In the regime of small \(\varepsilon\), the overlap region \(\|x\| \in [1-\varepsilon,\,1+\varepsilon]\) shrinks to a point and the quantity \(F_\varepsilon(a,\varphi)\) decays to zero as a function of \(\varepsilon\). Direct evaluation of the exact representation in \cref{eq:Fun_M} then suffers from severe cancellation and loss of precision. Thus, it is convenient to introduce a rescaled parameter that resolves the behaviour near \(a = 1\) as \(\varepsilon \downarrow 0\), and to work with a quantity that captures the essential \(\varepsilon\)-dependence in a unified way.

We reparametrise the interval \([1-\varepsilon,\,1+\varepsilon]\) by
\begin{equation}
\label{eq:Interval_Change_of_Variables}
	\lambda
	\coloneqq
	\frac{a - 1}{2\varepsilon} + \frac{1}{2},
	\qquad
	a = a(\lambda,\varepsilon)
	\coloneqq
	1 - (1-2\lambda)\varepsilon,
	\qquad
	\lambda \in [0,1].
\end{equation}
On the subinterval where \(a \leq 1\), the function \(F_\varepsilon\) admits the representation \cref{eq:Simple_F_first_branch_G_and_Quad}. This suggests introducing the reparametrised quantity
\begin{equation}\label{eq:def_H}
	H(\lambda,\varepsilon)
	\coloneqq
	\frac{1}{8\pi} \del{ G\del{a(\lambda,\varepsilon),\Phi\del{ \varphi_\varepsilon 
	\big( a(\lambda,\varepsilon) \big)
	} }
		- \pi\del{ 1-a(\lambda,\varepsilon)^{2} }
		},
	\qquad
	\lambda \in [0,1].
\end{equation}
For \(a \leq 1\) this coincides with \(F_\varepsilon(a,\varphi)\), while for \(a > 1\) it still encodes the nontrivial contribution of \(G\) that appears in the remaining branches of \cref{eq:Fun_M}.  An asymptotic expansion of \(H(\lambda,\varepsilon)\) as \(\varepsilon \to 0\) therefore yields, in a single framework, the leading-order behaviour of \(G(a,\Phi(\varphi))\), and hence of
\(F_\varepsilon\), throughout the overlap region.

The next lemma summarises the asymptotic structure of \(H\) for small \(\varepsilon\), distinguishing the two geometric regimes that arise in \cref{eq:Intersection_Angle_Cos}. A refined expansion at the special point \(a = 1\) (where \(F_\varepsilon\) vanishes cubically in \(\varepsilon\)) will be discussed subsequently.

\begin{lemma}\label{lem:H_asymptotic}
	Let \( 0 < \varepsilon \ll 1\), and \( \lambda \in [0,1]\). Parametrising \(a = 1 - (1-2\lambda)\varepsilon \in [1-\varepsilon, 1+\varepsilon]\), 
	%then we can write
	the function \(H\) from \cref{eq:def_H} admits the expansion
	\begin{equation}
	\label{eq:Assymptotic_expansion_of_h_two_branches}
		%F_\varepsilon(a,\varepsilon) =
		H(\lambda,\varepsilon) =
		\begin{dcases*} 
			h_1^{(1)} (\lambda) \varepsilon^2 \log \varepsilon^2 + h_2^{(1)} (\lambda) \varepsilon^2 + O(\varepsilon^3)
			& if \( a^2 \leq 1 + \varepsilon^2\),
			\\
			h_0^{(2)}(\lambda) + h_1^{(2)}(\lambda) \varepsilon + h_2^{(2)}(\lambda) \varepsilon^2 + O(\varepsilon^3)
			& if \( a^2 > 1 + \varepsilon^2\);
		\end{dcases*}
	\end{equation}
	where, defining \( \omega_\pm = \arccos( \pm(1-2\lambda) )\), we have
	\begin{align*}
		h_1^{(1)} (\lambda) &= \frac{1}{4\pi} (1-2\lambda) \sqrt{ \lambda (1-\lambda) },
		\qquad
		h_2^{(1)}(\lambda) = \frac{1-2\lambda}{4\pi} \left[ (1-2\lambda) \omega_+ - 3\sqrt{\lambda (1-\lambda)} \right],
		\\
		h_0^{(2)} (\lambda) &= \frac{1}{4\pi} \sbr{ \Im \Li(-e^{2 i \omega_-}) + 4(2\lambda -1) \del{1 - \log(4\lambda - 2) } \sqrt{\lambda (1-\lambda)} },
		\\
		h_1^{(2)} (\lambda) &= \frac{2\lambda - 1}{4\pi} \sbr{ (\pi - 2 \omega_-) + 2(2\lambda -1) \del{1 - \log(4\lambda - 2) } \sqrt{\lambda (1-\lambda)} },
		\\
		h_2^{(2)} (\lambda) &= \frac{(2\lambda - 1)^2}{8\pi} \sbr{ \pi - (2\lambda-1) \del{1 + 2 \log(4\lambda - 2) } \sqrt{\lambda (1-\lambda)} }.
	\end{align*}
\end{lemma}
The proof is given in \cref{App:Proof_A}. Extended precision evaluations of \(H(\lambda,\varepsilon)\) and its asymptotic expansion\footnote{Relevant notebooks in repository: (1) \texttt{\color{xkcdOceanBlue}4 - Tests H - Branch 1.ipynb} and (2) \texttt{\color{xkcdOceanBlue}5 - Tests H - Branch 2.ipynb} compute \(H\) and its asymptotic approximation in high precision for \( a^2 \leq 1 + \varepsilon^2\) and \( a^2 > 1 + \varepsilon^2\), respectively. (3) \texttt{\color{xkcdOceanBlue}6 - H - High double.ipynb} provides a double-precision implementation.} yield the error curves shown in \cref{fig:Asymptotics_Error}. There we present, on a logarithmic scale, the absolute differences between \(H\) and its asymptotic expansion for the family of values \( \varepsilon \in \{10^{-2 - 4a}\}_{a \in \llb 1,7 \rrb} \). We observe that both approximations exhibit near-constant errors that remain at a fixed fraction of \(\varepsilon\). Using this observation, we observed that the maximum error for each curve behaves as \( \varepsilon \times 10^{-2}\). The accuracy of the approximation allows us to provide a further stable double-precision approximation of \(H\) by identifying \(a\) with its counterpart branch in terms of \(\lambda\). 
Consequently, the simple expansion 
\(
	\frac{1}{8} (1-a^2) = \frac{1}{4} (1-2\lambda) \varepsilon - \frac{1}{8} (1-2\lambda)^2 \varepsilon^2 
\)
allows us to recover \(G(a,\Phi(\varphi))\) from \(H\) in a numerically stable way, and hence to
obtain robust double-precision implementations of \(F_\varepsilon\) in the overlap region.

% ------------------------------------- %
\begin{figure}[h!]
\centering
\captionsetup[subfloat]{format=hang}
	\subcaptionbox*{}[7cm]{\includegraphics[scale=0.7]{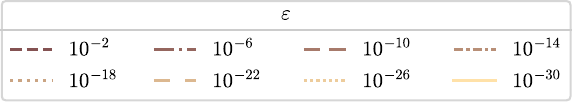}}
	\\[-0.5em]
\hspace{-1.25cm}
    \subcaptionbox{Values on the branch \( a^2 \leq 1+\varepsilon^2 \).}[6cm]{
    \includegraphics[scale=0.55]{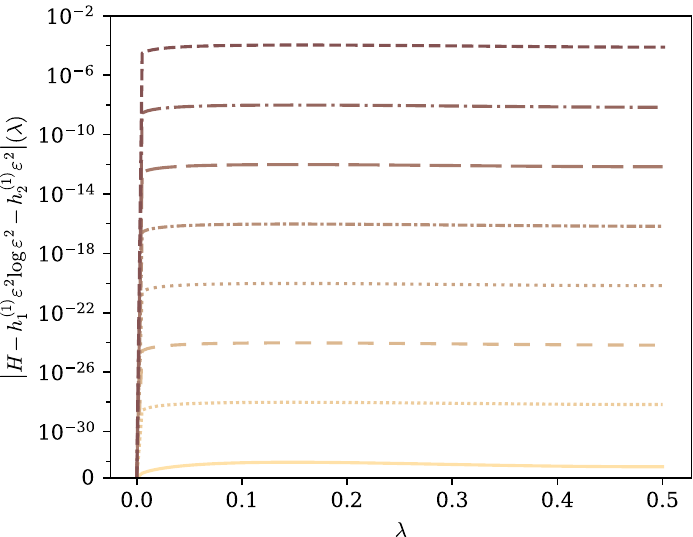}
	}
	\hspace{1.25cm}
    \subcaptionbox{Values on the branch \( a^2 > 1+\varepsilon^2 \).}[6cm]{
    \includegraphics[scale=0.55]{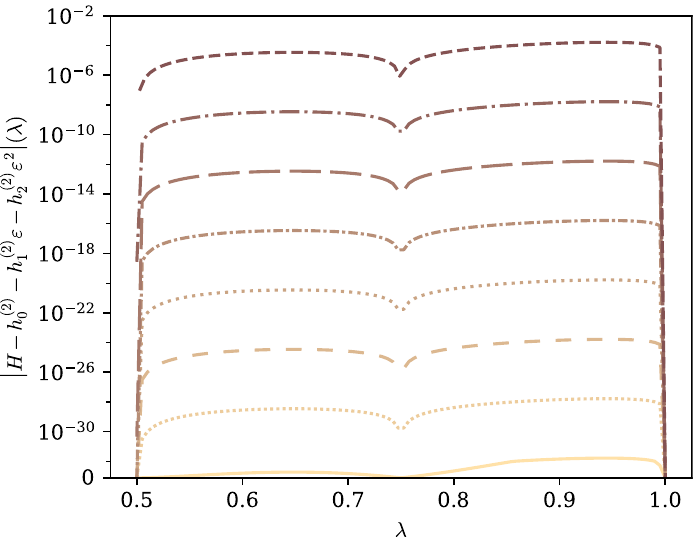}
	}
	%
	%
    %\vspace{1em}
    \caption{Absolute error of \(H\) against the asymptotic approximation from Lemma \ref{lem:H_asymptotic} for decreasing values of \(\varepsilon\).}
\label{fig:Asymptotics_Error}
%\vspace{-1\baselineskip}
\end{figure}
% ------------------------------------- %

A particular value of interest occurs at the boundary \(a = 1\), where two branches of \(F_\varepsilon\) coincide by continuity. In this setting, it is possible to determine a specialised asymptotic expansion in terms of \(\varepsilon\) that reveals a cubic decay. This is the purpose of the next lemma.

\begin{lemma}\label{lem:H_asymptotic_at_1}
	Let \( 0 < \varepsilon \ll 1\). We can write the expansion
	\begin{equation}
	\label{eq:Expansion_of_F_1}
		F_\varepsilon(1,\varepsilon) =
		\frac{1}{8\pi} G\del{1; \Phi(\varphi)}
		=
		%\frac{1}{8\pi} \del{ \frac{6}{18} \varepsilon^3 \log \varepsilon - \frac{5}{18} \varepsilon^3 } 
		\frac{1}{144\pi} \del{ 2 \varepsilon^3 \log \varepsilon^3 - 5 \varepsilon^3 } 
		+ O(\varepsilon^5 \log \varepsilon) .
	\end{equation}
\end{lemma}

A proof is presented at \cref{App:Proof_B}, where we directly expand the integral form \cref{eq:alternative_H_general} in terms of \(\varepsilon\). We tested the quality of the approximation\footnote{Relevant notebook in repository:  \texttt{\color{xkcdOceanBlue}7 - G - Computation.ipynb}.} and obtained an absolute error that decreased as \( \varepsilon^5 \times 10^{-1}\). 

The last result we will present in this section follows from the proof of Lemma \ref{lem:H_asymptotic}, which allows us to provide a high-precision evaluation of \(E_\varepsilon\):

\begin{corollary}\label{lem:phi_asymptotics}
	Let \( 0 < \varepsilon \ll 1\). We can write the expansion
	\begin{equation}
		\varphi = \varphi_\varepsilon (a(\lambda,\varepsilon) ) = %\arccos \frac{1 - a^2 - \varepsilon^2}{2a\varepsilon} 
		\arccos(1-2\lambda) + \sqrt{ \lambda(1-\lambda)} \varepsilon + \frac{3}{4} (1-2\lambda) \sqrt{ \lambda (1-\lambda) } \varepsilon^2 + O(\varepsilon^3).
	\end{equation}
\end{corollary}

% ––––––––––––––––––––––––––––––––––––– %
% ––––––––––––––––––––––––––––––––––––– %
\subsection{Numerical experiments}\label{sec:Numerics}

We conclude this note by describing the numerical implementation of \(F_\varepsilon\) and \(E_\varepsilon\) and illustrating the asymptotic behaviour predicted by our analyses. Throughout this subsection, we use the rescaling introduced in \cref{eq:Interval_Change_of_Variables}. All experiments were performed on a MacBook Pro 2020 with an M1 chip and 16 GB of RAM. The code for this and the previous sections was written in Python \texttt{3.10.14} and primarily relies on NumPy \texttt{2.2.5}, SciPy \texttt{1.15.1}, SymPy \texttt{1.3.3}, and Mpmath \texttt{1.3.0}.

% ------------------------------------- %
\begin{figure}[h!]
\centering
	\includegraphics[scale=0.6, page=1]{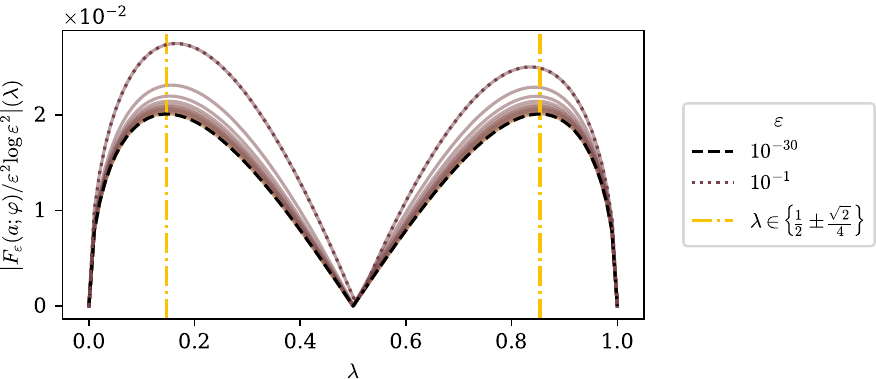}
	\caption{Extended-precision plot of \(F_\varepsilon \) scaled by \( \varepsilon^2 \log \varepsilon^2\) under the parametrisation \cref{eq:Interval_Change_of_Variables}. The function values corresponding to \(\varepsilon\) taking the values \num{e-1} and \num{e-30} are represented by dotted and dashed lines, respectively. The asymptotic maxima are marked with dotted-dash vertical lines.}
\label{fig:Asymptotics_F}
%\vspace{-1\baselineskip}
\end{figure}
% ------------------------------------- %

For ease of presentation, we introduce the scaled quantity \(J_\varepsilon\) as the rescaling of \(F_\varepsilon(a;\varphi_\varepsilon(a))\) by the reciprocal of \( \varepsilon^2 \log \varepsilon^2\). Whenever we use the parametrisation \cref{eq:Interval_Change_of_Variables}, we understand \(J_\varepsilon\) to be given by
\begin{equation}
\label{eq:def_J}
	J_\varepsilon (\lambda) = F_\varepsilon \del{ a(\lambda, \varepsilon); \varphi_\varepsilon \del{a(\lambda,\varepsilon)} } \big/ \varepsilon^2 \log \varepsilon^2.
\end{equation}

We first implemented an extended-precision version of \(F_\varepsilon\) with up to 150 digits and evaluated \(J_\varepsilon\) for \(\varepsilon \in \{10^{-k}\}_{k \in \llb 1,30 \rrb}\). The resulting curves are shown in \cref{fig:Asymptotics_F}. For moderate values of \(\varepsilon\), the scaled profiles are visibly asymmetric, but as \(\varepsilon \to 0\) they collapse onto an asymptotically symmetric limit curve described by the reflection of \(h_1^{(1)}\) (see \cref{eq:Assymptotic_expansion_of_h_two_branches}) about \( \lambda^\dagger \coloneqq \sfrac 1 2\).

This loss of asymmetry can be quantified by the scaled uniform norm of the antisymmetric part of \(J_\varepsilon\) around \(\lambda^\dagger\). We define the asymmetry index
\[ 
	\eta(\varepsilon) 
	\coloneqq  
	\big\| J_\varepsilon (\lambda) - J_\varepsilon (1-\lambda) \big\|_\infty
	=
	\max_{\lambda \in [0,1]} 
	\big| 
	%F_\varepsilon \del{ a(\lambda, \varepsilon); \varphi_\varepsilon \del{a(\lambda,\varepsilon)} } - F_\varepsilon \del{ a(1-\lambda, \varepsilon); \varphi_\varepsilon \del{a(1-\lambda,\varepsilon)} } 
	F_\varepsilon \del{ a; \varphi} (\lambda) - F_\varepsilon \del{ a; \varphi} (1-\lambda) 
	\big| \big/ \varepsilon^2 |\log \varepsilon^2|.
\]
An ordinary regression on a logarithmic scale suggests that \( \eta(\varepsilon) \sim 2\varepsilon \times \num{e-2}\), so that the rate of convergence, under the uniform norm, of the original (unscaled) quantity is about \(\varepsilon^3 \, |\log \varepsilon^4|\). As a consequence, the two critical points of \(h_1^{(1)}\) become the asymptotic minima of \(J_\varepsilon\).

Furthermore, we observe three local minimisers of \(J_\varepsilon\) located at the endpoints of the interval \(\lambda \in \{0,1\}\) and at \(\lambda^\dagger\). The latter minimiser corresponds to the value \(a = 1\). This is consistent with the geometric definition of \(F_\varepsilon\) as the integral of a signed kernel over an intersection: at \(a = 1\), the configuration produces the most symmetric overlap, where the \(\varepsilon\)-disc straddles the unit disc and the logarithmic kernel \(k\) changes sign across the interface.

% ------------------------------------- %
\begin{figure}[h!]
\centering
	\includegraphics[scale=0.6, page=1]{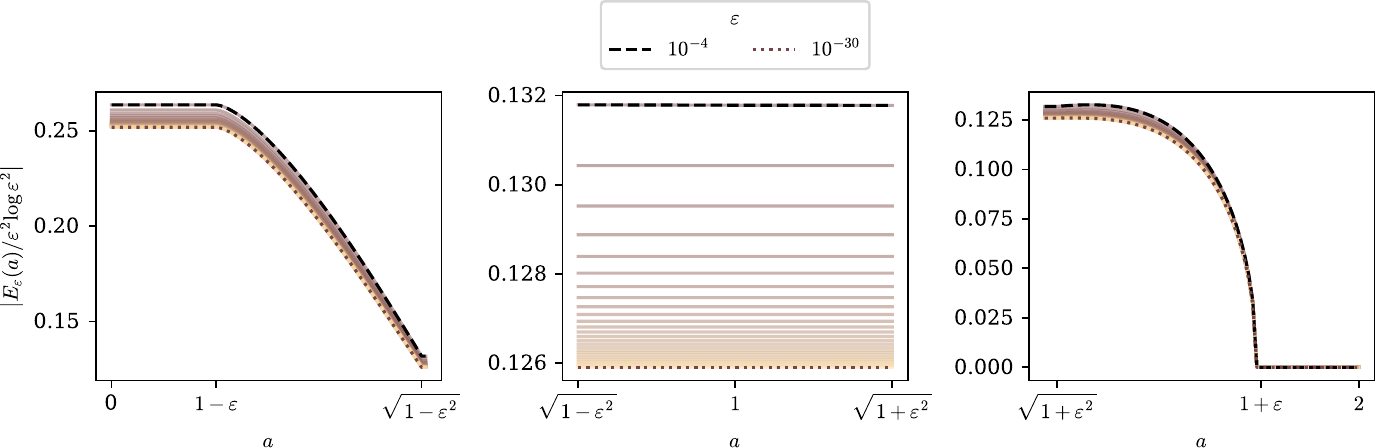}
	\caption{High precision plots of \(E_\varepsilon \) scaled by \( \varepsilon^2 \log \varepsilon^2\) for \(a\) inside the intervals \([0,\sqrt{1-\varepsilon^2}]\), \([\sqrt{1-\varepsilon^2},\sqrt{1+\varepsilon^2}]\), and \([\sqrt{1+\varepsilon^2},2]\). The function values corresponding to \(\varepsilon\) taking the values \num{e-4} and \num{e-30} are represented by dotted and dashed lines, respectively.}
\label{fig:EP_E}
%\vspace{-1\baselineskip}
\end{figure}
% ------------------------------------- %

Our results suggest that for \(\varepsilon \lesssim 10^{-5}\), the function \(F_\varepsilon\) is already well approximated by \(h_1^{(1)}\) in double-precision arithmetic. We used this approximation in our double-precision implementation\footnote{Relevant notebook in the repository: \texttt{\color{xkcdOceanBlue}8 - E - High double.ipynb}.} of \(E_\varepsilon\) for \(\varepsilon \leq \num{e-5}\). This allowed us to reproduce \cref{fig:EP_E}, originally generated in extended precision, down to \(\varepsilon = \num{e-14}\).

Comparing the extended- and double-precision versions of \cref{eq:def_J}, we observe that the double-precision evaluation of the scaled quantity agrees with the extended-precision reference to within about \(\num{e-7}\) for \(\varepsilon \gtrsim \num{e-10}\), and the discrepancy remains below \(\num{e-3}\), where a loss of precision becomes apparent for very small \(\varepsilon\). When this difference is translated back to the unscaled quantity \(E_\varepsilon\), the corresponding absolute error lies between \(\num{e-15}\) and \(\num{e-29}\), i.e., at or below the level expected from double-precision roundoff. Thus, our implementation captures \(E_\varepsilon\) with essentially machine accuracy over the entire range of tests; the apparent loss of digits in the scaled quantity for very small \(\varepsilon\) is merely a consequence of the ill-conditioning introduced by the factor \(\varepsilon^{2}\log\varepsilon^{2}\).

Additional experiments are available in our repository to further validate the results of this note. In particular, we also provide a Python module and a \textsc{Matlab} implementation for evaluating \(E_\varepsilon\) as a function of \(a\).

% ––––––––––––––––––––––––––––––––––––––––––––––––––––––––––––––––––––––––––––––––– %

\subsection*{Declarations}

\subsubsection*{Funding}
%A.M.-T. acknowledges support of MAC-MIGS CDT Scholarship under EPSRC grant EP/S023291/1.
This research was supported by the Engineering and Physical Sciences Research Council (EPSRC) via the Postdoctoral Pathways for Growth in the UK scheme, administered by the University of Edinburgh.

\subsubsection*{Acknowledgements}
The work was done partially while the author was participating in the programme Multiscale Analysis and Methods for Quantum and Kinetic Problems at the Institute for Mathematical Sciences, National University of Singapore, in 2023.

\subsubsection*{Data availability statement}

The data that support the findings of this study are openly available in \texttt{DiscConv}. \href{https://doi.org/10.5281/zenodo.7819767}{doi: 10.5281/zeno\-do.17807158}.

\subsubsection*{Uses of generative artificial intelligence}
%Declaration of generative AI and AI-assisted technologies in the writing process}

During the preparation of this manuscript, the author used large language models (in particular ChatGPT-v5.1 and DeepSeek-v3) for assistance with language editing and with the manipulation and checking of certain integral expressions (GAIDeT taxonomy (2025): Proofreading and editing \& Preliminary hypothesis testing). After using this tool/service, the author reviewed and edited the content as needed and takes full responsibility for the content of the publication. All mathematical arguments, results, and numerical experiments were formulated, verified, and validated by the author. GAI tools are not listed as authors and do not bear responsibility for the final outcomes.

\subsubsection*{Conflicts of interest}
The author does not work for, advise, own shares in, or receive funds from any organisation that could benefit from this article, and has declared no affiliations other than their research organisation.

% ––––––––––––––––––––––––––––––––––––––––––––––––––––––––––––––––––––––––––––––––– %
\addcontentsline{toc}{part}{Bibliography}
\setlength{\bibsep}{0pt plus 0.3ex}
\footnotesize
\bibliographystyle{siamplain}
%\bibliography{cites_nodi}
\bibliography{Main.bbl}

% ––––––––––––––––––––––––––––––––––––––––––––––––––––––––––––––––––––––––––––––––– %
% ––––––––––––––––––––––––––––––––––––––––––––––––––––––––––––––––––––––––––––––––– %

\newpage
\begin{appendices}
\renewcommand{\theequation}{A--\arabic{section}.\arabic{equation}}
\setcounter{equation}{0}

% ––––––––––––––––––––––––––––––––––––––––––––––––––––––––––––––––––––––––––––––––– %
\section{Proof of Lemma \ref{lem:H_asymptotic}}\label{App:Proof_A}

\subsection*{First branch}

Let's start with the assumption that \( a^2 \leq 1 + \varepsilon^2\), which translates to \( \lambda \leq  \frac{ (\varepsilon-1) + \sqrt{1+\varepsilon^2} }{2\varepsilon}\).
Under the parametrisation by \(\lambda\), we have that \cref{eq:H_Full} can be written as
\begin{equation}\label{eq:H_lambda_split_1}
	H(\lambda,\varepsilon) %\coloneqq H(a,\varphi)
	= \frac{1}{4\pi} [\mathfrak{A} + \mathfrak{B} + \mathfrak{C}] - \mathfrak{D}; %\frac{1}{8} (1-a^2);
\end{equation}
where for \( \phi = \dfrac{1}{2} \arccos \left( \dfrac{ \varepsilon^2  - 1 - a^2}{ 2 a }  \right) \) we have
\begin{align*}
	\mathfrak{A} &= \Im \big( \Li (-a e^{2i\phi}) \big),
	&
	\mathfrak{B} &= (1-a^2) \left[ \phi - \frac{1}{2}  \arctan\left(\frac{ a \sin 2\phi}{1 + a \cos 2\phi} \right)  \right],
	&
	\mathfrak{C} &= a ( 1 - \log \varepsilon  )\sin 2\phi,
	&
	\mathfrak{D} &= \frac{1}{8} (1-a^2).
\end{align*}

Let us introduce \( \beta \coloneqq 1 -2 \lambda\), for which \( a = 1 - \beta\varepsilon\) and \( S \coloneqq 2\sqrt{\lambda (1-\lambda)} = \sqrt{1 - \beta^2} \). Let us also define \( A \coloneqq \pi - 2 \phi\), then \( -a e^{2i \phi} = a e^{ -iA } \).
By definition \( \varepsilon^2 - 1 - a^2 = -2 + 2 \beta \varepsilon + (1-\beta^2) \varepsilon^2 \), and using \( (1-\delta)^{-1} = 1 + \delta + \delta^2 + O(\delta^3)\), we get
\[
	\cos A = -\cos 2\phi = 1 - \frac{S^2}{2 (1-\beta \varepsilon)} = 1 - \frac{S^2}{2} \varepsilon^2 - \frac{S^2 \beta}{2} \varepsilon^3 + O(\varepsilon^4).
\]
Matching this expansion with the series \( \cos \delta = 1 - (\sfrac 1 2) \delta^2 + (\sfrac 1 4) \delta^4 + O(\delta^6)\), we obtain %the following expansion:
%\(A = a_1 \varepsilon + a_2 \varepsilon^2+ O(\varepsilon^3) \):
%\begin{align*}
%	%A^2 = a_1^2 \varepsilon^2 + 2 a_1 a_2 \varepsilon^3 + a_2^2 \varepsilon^4 + O(\varepsilon^4)
%	\frac{S^2}{2} &= \frac{a_1^2}{2} 	&\Rightarrow\qquad a_1 &= S,	\\
%	a_1 a_2 &= \frac{S^2 \beta}{2}  		&\Rightarrow\qquad a_1 &= \frac{S\beta}{2},
%	% Thus:
%	A = S \varepsilon + \frac{S\beta}{2} \varepsilon^2 + O(\varepsilon^3).
%\end{align*}
that \( A = S \varepsilon + (\sfrac{1}{2}) S\beta \varepsilon^2 + O(\varepsilon^3) \).
Now, the Taylor series of sine gives
\begin{equation}
\label{eq:angle_cos_sine_DS}
	\phi = \frac{\pi}{2} - \frac{1}{2} S \varepsilon - \frac{1}{4} S\beta \varepsilon^2 + O(\varepsilon^3),
	\quad
	\cos 2\phi = -1 + \frac{1}{2} S^2 \varepsilon^2 + O(\varepsilon^3),
	\quad\text{and}\quad
	\sin 2\phi = S\varepsilon + \frac{1}{2} S\beta \varepsilon^2 + O(\varepsilon^3).
\end{equation}
We automatically obtain that
\(
	a \sin 2\phi = S \varepsilon - \frac{1}{2} S \beta \varepsilon^2 + O(\varepsilon^3).
\)
and thus
\[
	\boxed{
	\mathfrak C = 
	%(-S) \varepsilon \log \varepsilon + (S)\varepsilon + \frac{1}{2} (S\beta) \varepsilon^2 \log\varepsilon  - \frac{1}{2} (S\beta) \varepsilon^2 
	S(1-\log\varepsilon) \varepsilon + \frac{1}{2} S\beta (\log \varepsilon - 1) \varepsilon^2 + O(\varepsilon^3)
	}
	\qquad\text{and}\qquad
	\boxed{
	\mathfrak{D} = \frac{1}{8} (1-a^2) = %\frac{1}{8} \beta \varepsilon \del{ 2 - \beta \varepsilon}.
	\frac{1}{4} \beta \varepsilon  - \frac{1}{8} (\beta \varepsilon)^2.
	}
\]

\subsubsection*{Expanding \(\mathfrak B\)}
% arctan

Observe that \( 1-a^2  = 2\beta \varepsilon - \beta^2 \varepsilon^2\). Now recall that \( \varepsilon^2 = 1 + a^2 + 2a \cos 2\phi\) in this branch. Let \( w \coloneqq 1 + a e^{2i \phi} = 1- ae^{-iA} \) and notice that
\begin{equation}\label{eq:A1_atan_as_arg}
	T_1(\varepsilon) \coloneqq  \arctan \frac{a \sin 2\phi}{1 + a \cos 2\phi} = \arg w
	\qquad\text{and}\qquad
	|w| = \sqrt{1 + a^2 + 2a \cos 2\phi} = \varepsilon.
\end{equation}
By definition 
\begin{align*}
	e^{2 i\phi} 
	%&= \cos 2\phi + i \sin 2\phi
	&=
	\del{-1 + \frac{1}{2} S^2 \varepsilon^2 }
	+ i \del{S\varepsilon + \frac{1}{2} S\beta \varepsilon^2} + O(\varepsilon^3),
	\\
	a e^{2 i\phi}  &= \sbr{ (-1 + \beta \varepsilon) + \frac{1}{2} S^2 \varepsilon^2} 
	+ i \del{ S\varepsilon - \frac{1}{2} S\beta \varepsilon^2  }   + O(\varepsilon^3),
	&
	w &= \varepsilon \del{ (\beta + iS) + \frac{1}{2} \del{ S^2  - i S\beta }\varepsilon  }  + O(\varepsilon^3).
\end{align*}
Here, let us define \( m_1 \coloneqq \beta + i S \) and \( m_2  \coloneqq (\sfrac{1}{2}) (S^2  - i \beta S) \) and write \( w = \varepsilon \del{m_1 + m_2 \varepsilon + O(\varepsilon^2) } \). Recalling the identity \( \arg z = \Im \log z\), we obtain that
\[
	\arg w = \arg m_1 \del{ 1 + \frac{m_2}{m_1} \varepsilon + O(\varepsilon^2) } 
	= \arg m_1 + \varepsilon \Im \frac{m_2}{m_1} + O(\varepsilon^2).
\]
Now, \(|m_1| = 1\) and \(  \sfrac{m_2}{m_1} = -(\sfrac 1 2) i S  \). As a result
\(
	T_1(\varepsilon) = \arg (\beta + iS) - \frac{1}{2} S \varepsilon + O(\varepsilon^2).
\)
Thence,
\[
	T_2(\varepsilon) \coloneqq \phi - \frac{1}{2} T_1(\varepsilon)
	%\\
	=  \frac{1}{2} \del{ \pi -  \arg (\beta + iS) } 
	%+ \del{\frac{1}{4} S (1+\beta^2) - \frac{1}{2} S } \varepsilon
	- \frac{1}{4} S \varepsilon - \frac{1}{4} S\beta \varepsilon^2
	+ O(\varepsilon^2).
\]
Here, considering the fact that \( \arg (\beta + iS) = \arctan \sfrac{S}{\beta} = \arccos\beta\), we obtain that
\[
	\boxed{
		\mathfrak{B} =
		\beta \del{ \pi -  \arccos\beta } \varepsilon
		-
		\frac{1}{2} \del{S\beta  + \beta^2  \del{ \pi -  \arccos\beta } } \varepsilon^2
		+ O(\varepsilon^3).
	}
\]

% Dilog
\subsubsection*{Expanding \(\mathfrak A\)}

We will use a similar expansion for the dilogarithm. First observe,
\begin{align*}
	\log w 
	%&= \log \varepsilon + \log m_1 + \log \del{ 1 + \frac{m_2}{m_1} \varepsilon + O(\varepsilon^2) }
	%= \log \varepsilon + i \arg m_1 + \varepsilon \frac{m_2}{m_1} + O(\varepsilon^2)
	%\\
	%&
	= \log \varepsilon + i \arccos \beta + \varepsilon \frac{m_2}{m_1} + O(\varepsilon^2)
	%= \log \varepsilon + i \arccos \beta - \frac{1}{2} \del{ \beta S^2 + i  S^3 + 2\beta^2 S }\varepsilon  + O(\varepsilon^2).
\end{align*}
and define \( \omega \coloneqq \arccos \beta\).
Now, combining the reflection formula of the dilogarithm \( \Li(1-\delta) = (\sfrac{\pi^2}{6}) - \log(1-\delta) \log \delta - \Li(\delta)\), and the series expansions \( \log(1-\delta) = -\delta + (\sfrac 1 2) \delta^2 + O(\delta^3)\) and \( \Li (\delta) = \delta + (\sfrac 1 4) \delta^2 + O(\delta^3)\), we obtain
\begin{align*}
	\Li(1-w) 
	%&= \frac{\pi^2}{6} - \log w \log(1-w) - \Li(w) = \frac{\pi^2}{6} - \log w \del{ -w + \frac{1}{2} w^2 + O(w^3) } - \Li(w)
	%\\
	&= \frac{\pi^2}{6} + w \log w - \frac{1}{2} w^2 \log w  - w - \frac{1}{4} w^2 + O(w^3 \log w ).
\end{align*}
We need the imaginary part of each term, where
\begin{align*}
	w \log w 
	%= \varepsilon \del{m_1 + m_2 \varepsilon + O(\varepsilon^2) } \del{ \log \varepsilon + i \arccos \beta + \frac{m_2}{m_1} \varepsilon  + O(\varepsilon^2) }
	&= m_1 \varepsilon (\log \varepsilon  + i \omega) + m_2 \varepsilon^2  (\log \varepsilon  + i \omega) + (m_1 \varepsilon) \frac{m_2}{m_1} \varepsilon + O(\varepsilon^3 \log \varepsilon),
	\\
	\Im w \log w 
	&= S \varepsilon \log \varepsilon + \beta \omega \varepsilon - \frac{1}{2} \beta S \varepsilon^2 \log\varepsilon 
	+ \frac{1}{2} \del{ S^2 \omega  - S\beta} \varepsilon^2 + O(\varepsilon^3 \log \varepsilon),
	\\
	-\Im w &= - S \varepsilon + \frac{1}{2} \beta S \varepsilon^2 + O(\varepsilon^3),
	\\
	-\dfrac{1}{4} \Im w^2 &
	%= -\dfrac{1}{4}  \Im m_1^2 \varepsilon^2 + O(\varepsilon^3) 
	= -\dfrac{1}{4}  \Im \del{ (\beta^2 - S^2) + i 2 \beta S } \varepsilon^2 + O(\varepsilon^3)
	= -\dfrac{1}{2}  \beta S \varepsilon^2 + O(\varepsilon^3),
	\\
	\frac{1}{2}\Im w^2 \log w 
	&= \frac{1}{2} \varepsilon^2 \del{ \log \varepsilon \Im m_1^2 + \omega \Re m_1^2 } + O(\varepsilon^3 \log\varepsilon)
	= \varepsilon^2 \del{ \beta S \log \varepsilon + \frac{1}{2} (\beta^2 - S^2) \omega } + O(\varepsilon^3 \log\varepsilon).
\end{align*}
Thus
\[
	%\Im \Li(1-w) =
	\boxed{
		\mathfrak{A} =  S \varepsilon \log \varepsilon + \del{\beta \omega -S} \varepsilon + \frac{1}{2} S \beta \varepsilon^2 \log \varepsilon
		+
		\frac{1}{2} \del{ \beta^2 \omega  - S\beta } \varepsilon^2
		 + O(\varepsilon^3 \log\varepsilon).
	}
	%\frac{1}{2} \del{ S^2 \omega  - S\beta} +  \frac{1}{2} \beta S  -\dfrac{1}{2}  \beta S  +  \del{ \frac{1}{2} (\beta^2 - S^2) \omega }
	%\frac{1}{2} \del{ S^2 \omega  - S\beta} +  \del{ \frac{1}{2} (\beta^2 - S^2) \omega }
	%\del{ S^2 \omega  - S\beta} +   (\beta^2 - S^2) \omega 
	%- S\beta +   \beta^2 \omega 
\]

Denote by \( \Pi_{f(\varepsilon^n)} \) the \(n\)--th projector of a polynomial in terms of \(f(\varepsilon)\), then we have that
\begin{align*}
	\Pi_\varepsilon \del{ \mathfrak{B}+\mathfrak{C} - 4\pi \mathfrak{D} }  
	%&= \beta \pi - \beta\omega + S - \pi\beta
	&= S - \beta\omega,
	&
	\Pi_{\varepsilon \log\varepsilon} \del{ \mathfrak{B}+\mathfrak{C} - 4\pi \mathfrak{D} }  
	&= -S,
	\\
	\Pi_{\varepsilon^2} \del{ \mathfrak{B}+\mathfrak{C} - 4\pi \mathfrak{D} }  
	%&= -\frac{1}{2} \del{S\beta  + \beta^2  \del{ \pi -  \omega } }  - \frac{1}{2} S\beta + \frac{\pi}{2} \beta^2
	&= \frac{1}{2} (\beta^2 \omega - 2S\beta),
	&
	\Pi_{\varepsilon^2 \log\varepsilon} \del{ \mathfrak{B}+\mathfrak{C} - 4\pi \mathfrak{D} }  
	&= \frac{1}{2} S\beta .
\end{align*}
Observe that the \( \varepsilon\) and \( \varepsilon \log \varepsilon\) terms cancel out with their counterparts in \( \mathfrak{A}\); also
\begin{align*}
	\Pi_{\varepsilon^2 \log\varepsilon^2} \del{ \mathfrak{A} + \mathfrak{B}+\mathfrak{C} - 4\pi \mathfrak{D} }  
	&= 
	\frac{1}{2} S\beta
	= (1-2\lambda) \sqrt{\lambda (1-\lambda)} = 4\pi h_1^{(1)}(\lambda),
	\\
	\Pi_{\varepsilon^2} \del{ \mathfrak{A} + \mathfrak{B}+\mathfrak{C} - 4\pi \mathfrak{D} }  
	&=
	\beta^2 \omega - \frac{3}{2} S\beta 
	=
	(1-2\lambda) \sbr{ (1-2\lambda) \arccos(1-2\lambda) - 3 \sqrt{\lambda (1-\lambda)} }
	= 4\pi h_2^{(1)}(\lambda).
\end{align*}

\subsection*{Second branch}

Let us now turn to the case \( a^2 > 1 + \varepsilon^2\), which translates to \( \lambda >  \frac{ (\varepsilon-1) + \sqrt{1+\varepsilon^2} }{2\varepsilon} > \sfrac 1 2\). Once more, we write \cref{eq:H_Full} as
\begin{equation}\label{eq:H_lambda_split_2}
	H(\lambda,\varepsilon) \coloneqq H(a,\varphi)
	= \frac{1}{4\pi} [\mathfrak{A} + \mathfrak{B} + \mathfrak{C}] - \mathfrak{D}; 
\end{equation}
where 
\begin{align*}
	\phi &= \dfrac{1}{2} \arccos \left( \dfrac{ (a^2 - 1)^2 }{ 2 a \varepsilon^2 } - \frac{1+a^2}{2a}   \right),
	&
	\mathfrak{A} &= \Im \big( \Li (-a e^{2i\phi}) \big),
	\\
	\mathfrak{B} &= (1-a^2) \sbr{ \phi - \frac{1}{2}  \arctan\left(\frac{ a \sin 2\phi}{1 + a \cos 2\phi} \right)  },
	&
	\mathfrak{C} &= a \sbr{ 1 - \log \frac{a^2-1}{\varepsilon}   }\sin 2\phi,
	&
	\mathfrak{D} &= \frac{1}{8} (1-a^2).
\end{align*}
Now define \(\beta \coloneqq 2\lambda - 1\) and \( \omega \coloneqq \arccos \beta\). Notice that \( a = 1 + \beta \varepsilon\) and \(\beta \in (0,1]\); in particular, we have that \( \beta > \varepsilon^{-1} (\sqrt{1+\varepsilon^2} -1 ) \). We can directly obtain the expansions
\(
	%a^2 = 1 + 2\beta \varepsilon + \beta^2 \varepsilon^2
	%a^2 + 1 = 2 + 2\beta \varepsilon + \beta^2 \varepsilon^2
	a^2 - 1 = 2\beta \varepsilon + \beta^2 \varepsilon^2
	\),
	%\qquad\text{and}\qquad
	\(
	(a^2 - 1)^2 = 4\beta^2 \varepsilon^2 + 4 \beta^3 \varepsilon^3 + \beta^4 \varepsilon^4,
\)
and
\[
	\boxed{
	\mathfrak D = -\frac{1}{4} \beta \varepsilon - \frac{1}{8} \beta^2 \varepsilon^2.
	}
\]
Employing the series expansion of a rational quotient, we have 
\begin{align}
	\cos 2\phi 
	%&= \frac{ 4\beta^2 + 4\beta^3 \varepsilon + \beta^4 \varepsilon^2  }{ 2 (1+\beta\varepsilon) } - \frac{ 2 +2\beta\varepsilon + \beta^2 \varepsilon^2 }{2(1+\beta\varepsilon)}
	%\\
	&= \frac{1}{2} \sbr{ \del{ 4\beta^2 (1+\beta \varepsilon) + \beta^4 \varepsilon^2  } - \del{ 2(1+\beta\varepsilon) + \beta^2 \varepsilon^2 } }
	\sbr{ 1 - \beta \varepsilon +  \beta^2 \varepsilon^2  + O(\varepsilon^3) }
	\notag
	\\
	&= \sbr{ 2\beta^2 + \frac{1}{2} \beta^4 \varepsilon^2 } -
	\sbr{ 1 + \frac{1}{2} \beta^2 \varepsilon^2 } 
	+ O(\varepsilon^3)
	\notag
	\\
	&= (2\beta^2 - 1) + \frac{1}{2} \beta^2 (\beta^2 -1) \varepsilon^2 + O(\varepsilon^3).
	\label{eq:A1_b2_cos2phi_1}
\end{align}
Now, let us assume the expansion \( \phi = \phi_0 + \phi_1 \varepsilon + \phi_2 \varepsilon^2 + O(\varepsilon^3)\), we can use the Taylor expansion of cosine and sine to obtain
\begin{align*}
	\cos 2\phi &= \cos 2\phi_0 - 2 \phi_1 \sin 2\phi_0 \varepsilon - 2\sbr{\phi_2 \sin 2\phi_0 + \phi_1^2 \cos 2\phi_0  } \varepsilon^2 + O(\varepsilon^3),
	\\
	\sin 2\phi &= \sin 2\phi_0 + 2 \phi_1 \cos 2\phi_0 \varepsilon + 2\sbr{\phi_2 \cos 2\phi_0 - \phi_1^2 \sin 2\phi_0  } \varepsilon^2 + O(\varepsilon^3).
\end{align*}
Proceed to match terms against \cref{eq:A1_b2_cos2phi_1}. First,
\(
	\cos 2\phi_0  = 2\cos^2 \phi_0 - 1 =2\beta^2 - 1
\)
yields \( \phi_0 = \omega\), \(\sin \phi_0 = \sqrt{1-\beta^2}\), and \( \sin 2\phi_0 = 2\beta \sqrt{1-\beta^2} \), where the sign of the latter comes from the domain of \(\beta\).
Also \( \phi_1 = 0\) and then
\[
	%-2\phi_2 \sin 2\phi_0 = \frac{1}{2} \beta^2 (\beta^2 -1) 
	\phi_2 = \frac{1}{4} \frac{ \beta^2 (1 - \beta^2) }{ \sin 2\phi_0 }
	= \frac{1}{8} \frac{ \beta (1 - \beta^2) }{ \sqrt{1 - \beta^2} }
	= \frac{1}{8} \beta \sqrt{1-\beta^2} = \frac{1}{8} \beta^2 \tan \omega.
\]
Collecting terms, we have obtained the general expansions
\begin{equation}\label{eq:angle_cos_sine_2}
\begin{aligned}
	\phi = \omega + \frac{1}{8} \beta \sqrt{1-\beta^2} \varepsilon^2 + O(\varepsilon^3),
	\qquad
	\cos 2\phi &= (2\beta^2 - 1) + \frac{1}{2} \beta^2 (\beta^2 -1) \varepsilon^2 + O(\varepsilon^3),
	\\
	\sin 2\phi &= 2\beta \sqrt{1-\beta^2} + \frac{1}{4} \beta (2\beta^2 -1) \sqrt{1-\beta^2}  \varepsilon^2  + O(\varepsilon^3).
\end{aligned}
\end{equation}
Now we can proceed to expand each term in \(H\).

% Dilog
\subsubsection*{Expanding \(\mathfrak A\)}

We use the Taylor expansion:
\[
	\Im \Li (z) = \Im  \Li (w) - \Im \frac{\log (1-w)}{w} (z-w) + \frac{1}{2} \Im \del{ \frac{1}{w(1-w)} + \frac{\log(1-w)}{w^2} } (z-w)^2 + O\del{(z-w)^3}.
\]
Letting \( z \coloneqq -a e^{2i\phi} \), we obtain
\begin{align*}
	z &= -(1+ \beta \varepsilon) e^{2i ( \phi_0 + \phi_2 \varepsilon^2 + O(\varepsilon^3) )}
	= -(1+ \beta \varepsilon) e^{2i \phi_0} \sbr{ 1 + 2 i \phi_2 \varepsilon^2 + O(\varepsilon^3) } 
	%\\
	= -e^{2i \phi_0} \sbr{ 1 + \beta \varepsilon + 2 i \phi_2 \varepsilon^2 }  + O(\varepsilon^3).
\end{align*}
We let \( w \coloneqq -e^{2i \phi_0}\), then \( 1 - w = e^{i\phi_0} (e^{-i\phi_0} + e^{i\phi_0}) = 2\cos \phi_0 e^{i\phi_0}\), \( \arg w = \phi_0\), and
\begin{align*}
	\log (1- w) &= \log (2 \cos \phi_0) + i\phi_0 = \log 2 \beta + i \phi_0,
	&
	\frac{(z-w)^2}{w(1-w)} &= - \frac{e^{i\omega}}{2 \cos \omega} \beta^2 \varepsilon^2 + O(\varepsilon^3),
	\\
	\frac{z-w}{w} \log (1-w) 
	%&= \del{\frac{z}{w} - 1} \del{\log 2 \omega + i \phi_0}
	&= \del{ \beta \varepsilon + 2 i \phi_2 \varepsilon^2 }\del{\log 2 \beta + i \phi_0} + O(\varepsilon^3),
	&
	\frac{(z-w)^2}{w^2} \log(1-w) &= \beta^2 \del{ \log 2 \beta + i \phi_0 } \varepsilon^2  + O(\varepsilon^3).
\end{align*}
Then
\[
	%\Im \Li (z)
	\boxed{
	\mathfrak A = \Im  \Li (-e^{2i \omega}) - \beta \omega \varepsilon + \sbr{ -\frac{1}{4} \beta \sqrt{1 - \beta^2} \del{1 + \log 2\beta} + \frac{1}{2} \beta^2 \omega }\varepsilon^2
	+ O\del{\varepsilon^3}.
	}
\]

% atan
\subsubsection*{Expanding \(\mathfrak B\)}

Proceeding as in \cref{eq:A1_atan_as_arg}, we have that
\[
	\arctan \frac{a \sin 2\phi}{1 + a \cos 2\phi} = \frac{2 a \beta \sqrt{1-\beta^2}}{1 + a (2\beta^2 - 1)}
	= 
	\arg(1 + a e^{2i\phi}) \eqqcolon \arg v.
\]
Since 
\( 
	v = 1-z = (1 - e^{2\i \omega}) + e^{2i\omega} \sbr{ \beta \varepsilon + 2i \phi_2 \varepsilon^2 } + O(\varepsilon^3)
	= 2\beta e^{i\omega} + e^{2i\omega} \sbr{ \beta \varepsilon + 2i \phi_2 \varepsilon^2 } + O(\varepsilon^3)
\), we obtain
\begin{align*}
	\arg v &= \arg 2\beta e^{i\omega} + \varepsilon \Im \frac{e^{2i\omega} \beta}{ 2\beta e^{i\omega} } + O(\varepsilon^2)
	= \omega + \frac{1}{2} \sqrt{1 - \beta^2} \varepsilon + O(\varepsilon^2),
	\\
	\phi - \frac{1}{2} \arg v 
	&=
	\frac{1}{4} \del{ 2 \omega -  \sqrt{1 - \beta^2} \varepsilon } + O(\varepsilon^2) .
\end{align*}
Thus
\[
	\boxed{
	\mathfrak B 
	=
	%-2\beta \varepsilon - \beta^2 \varepsilon^2
	- \beta \omega \varepsilon + \frac{1}{2} \del{ \beta \sqrt{1 - \beta^2} - \beta^2 \omega } \varepsilon^2 + O(\varepsilon^3).
	}
\]

% log
\subsubsection*{Expanding \(\mathfrak C\)}

Here, we employ the Taylor expansion 
\( 
	\log \frac{a^2-1}{\varepsilon} = \log (2\beta + \beta^2 \varepsilon) 
	%= \log 2\beta +  \frac{\beta^2 \varepsilon}{2\beta} - \frac{\beta^4 \varepsilon^2}{8\beta^2} + O(\varepsilon^3)
	= \log 2\beta +  \frac{1}{2} (\beta \varepsilon) - \frac{1}{8} \beta^2 \varepsilon^2 + O(\varepsilon^3)
\)
which yields
\begin{align*}
	a - a\log \frac{a^2-1}{\varepsilon} 
	&= (1+\beta\varepsilon) \del{ 1 - \log 2\beta -  \frac{1}{2} (\beta \varepsilon) + \frac{1}{8} \beta^2 \varepsilon^2 + O(\varepsilon^3) }
	\\
	&= (1- \log 2\beta)  + \frac{1}{2} \del{ \beta - 2\beta \log 2\beta} \varepsilon - \frac{3}{8} \beta^2 \varepsilon^2 + O(\varepsilon^3).
\end{align*}
Finally, from \cref{eq:angle_cos_sine_2}, we obtain
\[
	\boxed{
	\mathfrak{C} = 2\beta (1- \log 2\beta) \sqrt{1-\beta^2} 
	+ \beta^2 \del{ 1 - 2 \log 2\beta} \sqrt{1-\beta^2} \varepsilon
	- \frac{1}{4} \beta \del{ 1 + \beta^2 + (2\beta^2-1) \log 2\beta  } \sqrt{1-\beta^2} \varepsilon^2 + O(\varepsilon^3).
	}
\]

Now we can collect terms:
\begin{align*}
	\Pi_{\varepsilon^0} \del{ \mathfrak{A}+\mathfrak{B}+\mathfrak{C} - 4\pi \mathfrak{D} }  
	&= \Im  \Li (-e^{2i \omega}) + 2\beta (1- \log 2\beta) \sqrt{1-\beta^2}
	\\
	&= 
	\Im  \Li (-e^{2i \omega}) + 4(2\lambda-1) (1- \log (4\lambda-2) ) \sqrt{\lambda (1-\lambda)}
	= 4\pi h_0^{(2)},
	\\
	\Pi_\varepsilon \del{ \mathfrak{A}+\mathfrak{B}+\mathfrak{C} - 4\pi \mathfrak{D} }  
	&=
	\beta (\pi - 2\omega) + \beta^2 ( 1 - 2 \log 2\beta) \sqrt{1-\beta^2}
	\\
	&= (2\lambda -1) \sbr{ (\pi - 2\omega) + 2(2\lambda -1)\del{ 1 - 2 \log (4\lambda -2 )} \sqrt{\lambda (1-\lambda)} }
	= 4\pi h_1^{(2)},
	\\
	\Pi_{\varepsilon^2} \del{ \mathfrak{A}+\mathfrak{B}+\mathfrak{C} - 4\pi \mathfrak{D} }  
	&=
	\frac{\pi}{2} \beta^2 - \frac{1}{4} \beta^3 \sqrt{1-\beta^2} (1 + 2\log 2\beta  )
	\\
	&=
	\frac{1}{2} (2\lambda-1)^2 \sbr{ \pi - (2\lambda-1) \del{1 + 2\log(4\lambda - 2) } \sqrt{\lambda (1-\lambda)} }
	= 4\pi h_2^{(2)}.
\end{align*}

% ––––––––––––––––––––––––––––––––––––––––––––––––––––––––––––––––––––––––––––––––– %
\section{Proof of Lemma \ref{lem:H_asymptotic_at_1}}\label{App:Proof_B}

%\begin{proof}
	Recall from \cref{eq:alternative_H_general} that
	\[
		(8 \pi) F_\varepsilon(1;\varphi) \eqqcolon f(\varepsilon) = \int\limits_{\sfrac \pi 2}^\varphi  s^2(\theta) \del{ 2 \log {s(\theta)} - 1 } \dif \theta.
	\]
	Here, we know that \( \cos \varphi = \sfrac{-\varepsilon}{2}\); thus, \( \varphi \in ( \sfrac{\pi}{2}, \sfrac{2\pi}{3} ) \) and \( s(\theta) = -2\cos \theta\) inside the integration region. Now, let us introduce the change of variables \( \theta  \equiv u + \sfrac{\pi}{2}\) which yields \( \cos\theta = -\sin u\) and thus
	\begin{align}
		f(\varepsilon) &= \int\limits_{0}^{ \gamma}  4 (\sin u)^2 \del{ \log \del{4 \sin^2 u} - 1 } \dif u
		= \int\limits_{0}^{ \gamma}  4 (\sin u)^2 \del{ (2\log 2 - 1) + 2 \log(\sin u) } \dif u
		\notag
		\\
		&= (8 \log 2 -4) \int\limits_{0}^{ \gamma}  \sin^2 u \dif u
		+ 8 \int\limits_0^\gamma \sin^2 u \log(\sin u) \dif u
		\eqqcolon (8 \log 2 -4) g_1(\varepsilon) + 8 g_2(\varepsilon)
		;
		\label{eq:f_as_two_integrals}
	\end{align}
	where \(  \gamma \coloneqq \arcsin \sfrac \varepsilon 2\).
	Recall the expansions \( \arcsin \delta = \delta + (\sfrac{1}{6}) \delta^3 + O(\delta^5)\) and \( \sqrt{1 - \delta^2} = 1 - (\sfrac 1 2) \delta^2 + O(\delta^4)\), then
	\begin{align}
		g_1(\varepsilon) = \int\limits_{0}^{ \gamma}  \sin^2 u \dif u 
		&= \frac{1}{4} (2\gamma - \sin 2\gamma )
		= \frac{1}{4} \del{2\gamma - \varepsilon \sqrt{1-\frac{\varepsilon^2}{4} }}
%		\\
%		&= \frac{1}{4} \del{2\del{ \frac{\varepsilon}{2} + \frac{1}{6} \frac{\varepsilon^3}{8} + O(\varepsilon^5) } - \varepsilon \del{ 1 - \frac{1}{2} \frac{\varepsilon^2}{4} + O(\varepsilon^4) } }
%		\\
%		&= \frac{1}{4} \del{  \varepsilon + \frac{1}{24} \varepsilon^3 + O(\varepsilon^5)  - \varepsilon + \frac{1}{8} \varepsilon^3 + O(\varepsilon^5) } 
		%\\&
		= \frac{1}{24} \varepsilon^3 + O(\varepsilon^5).
		\label{eq:expansion_f_1_int_1}
	\end{align}
	
	For the second integral in \cref{eq:f_as_two_integrals}, we recall that \( \lim\limits_{\varepsilon \downarrow 0} \varphi = \sfrac\pi 2\), for which the integration interval shrinks as \( \gamma \to 0\). In fact, \( g_2(\varepsilon)\) is clearly a differentiable function in \( [0,1]\). 
	Moreover, since \( (1-\delta)^{\sfrac{-1}{2}} =  \sum\limits_{k=0}^\infty 4^{-k} \binom{2k}{k} \delta^{k} \) for any \( |\delta| < 1\), we obtain
	\begin{align*}
		g_2'(\varepsilon) 
		%= \sin^2 \gamma \log (\sin \gamma) \od{ }{\varepsilon} \arcsin \frac{\varepsilon}{2}
		&= \frac{\varepsilon^2}{4 \sqrt{4-\varepsilon^2} } \log \frac{\varepsilon}{2}
		= \frac{1}{8} \varepsilon^2 \del{ 1 - \frac{\varepsilon^2}{4} }^{\sfrac{-1} 2} \log \frac{\varepsilon}{2}
		\\
		&= \frac{1}{8} \varepsilon^2 \del{ \sum_{k=0}^\infty \frac{(2k)!}{16^{k} (k!)^2}  \varepsilon^{2k} } \log \frac{\varepsilon}{2}
		= \sum_{k=0}^\infty c_{k}  \varepsilon^{2k+2}  \log \frac{\varepsilon}{2},
	\end{align*}
	where \( c_k =16^{-k}  (2k)! \big/ 8 (k!)^2\). Then, since \( g_2(0)=0\), we have that
	\begin{align}
		g_2(\varepsilon) = \int\limits_0^\varepsilon g_2'(t) \dif t
		&=\sum_{k=0}^\infty c_{k} \int\limits_0^\varepsilon t^{2k+2}  \log \frac{t}{2} \dif t
		=\sum_{k=0}^\infty c_{k} \del{ \frac{\varepsilon^{2k+3}}{2k+3} \del{ \log \frac{\varepsilon}{2} - \frac{1}{2k+3} } }
		\notag
		\\
		&= % (c_0 = \frac{1}{8})  \frac{1}{3} 
		\frac{1}{24} \varepsilon^3 \del{ \log \varepsilon - \frac{1}{3} - \log 2 }  + O(\varepsilon^5 \log \varepsilon).
		\label{eq:expansion_f_1_int_2}
	\end{align}
	
	Replacing \cref{eq:expansion_f_1_int_1} and \cref{eq:expansion_f_1_int_2} in \cref{eq:f_as_two_integrals}, we finally obtain
	\begin{align}
		f(\varepsilon)
		&=
		(2\log 2 -1) \frac{1}{6} \varepsilon^3 
		+ %O(\varepsilon^5) +
		\frac{1}{3} \varepsilon^3 \del{ \log \varepsilon - \frac{1}{3} - \log 2 }  + O(\varepsilon^5 \log \varepsilon)
		%\notag
		%\\
		%&
		=
		\frac{1}{3} \varepsilon^3 \del{ \log \varepsilon - \frac{5}{6} } + O(\varepsilon^5 \log \varepsilon).
		\label{eq:final_expansion_reduction}
	\end{align}
	%Compare \cref{eq:final_expansion_reduction} against \cref{eq:Expansion_of_F_1}.
%\end{proof}

% ––––––––––––––––––––––––––––––––––––––––––––––––––––––––––––––––––––––––––––––––– %
\section{Tables of relevant quantities and symbols}\label{App:Tables}

% ------------------------------------- %
\begin{table}[h!]
\setlength{\tabcolsep}{0.2em}
\centering
%\fontsize{6.5}{7.5}\selectfont
%\footnotesize
\fontsize{8}{9}\selectfont
\def\arraystretch{1.95}
\begin{tabular}{c c l c l c l c c c  l}
\toprule
{ \bf Symbol } 
&& 
\multicolumn{1}{c}{{\bf Meaning }}
\\
\midrule
	\(\disc\) && The closed Euclidean unit ball \( B[0;1]\).
 	\\
	\( \raisebox{.05em}{\smalldisc}  \) && The intersection set \( B[0;1] \cap B[x; \varepsilon] \).
	\\
	\( \circleddash\) && The translated region \(x - \raisebox{.05em}{\smalldisc}\).
	\\
	\( \rotatebox[origin=c]{-90}{\LEFTcircle} \) && The portion of \(\circleddash\) in the closed upper plane: 
	\(\circleddash \cap \R \times \R_{\geq 0}\).
	\\
\bottomrule
\end{tabular}
\\[0.5em]
%\footnotetext
%\flushleft
\caption{Graphical notation used throughout this note.}
\vspace{-0.5\baselineskip}
\label{tb:Symbols}
\end{table}
% ------------------------------------- %

% ------------------------------------- %
\begin{table}[h!]
\setlength{\tabcolsep}{0.4em}
\centering
%\fontsize{6.5}{7.5}\selectfont
%\footnotesize
\fontsize{8}{9}\selectfont
\def\arraystretch{1.75}
\begin{tabular}{c c c c c c c c c c  l}
\toprule
\multirow{2}{*}{ \bf Angle $\theta$ } 
&&  \multicolumn{7}{c}{\bf Quantities of interest }
&& \multirow{2}{*}{ \bf Comments } 
\\  \cline{3-9}
&& \(s(\theta)\) && \(L(\theta)\) && \( P(\theta) \) && \( \Phi(\theta) \)
%
%
%\\[-0.2em]
%
\\
\midrule
	%\(\theta\) && \(-a \cos \theta + \sqrt{1 - a^2 \sin^2 \theta} \) && \( \frac{s^2(\theta) - 1 - a^2}{2a} \) && \( (\sfrac 1 2) \arccos L(\theta) \)
	%\\
	0	&& 	\(1-a\)	&&	\(-1\) && \(2\log |1-a|\)	&&	\( \frac \pi 2\)
	\\
	%\(\sfrac{\pi}{2}\)	&& 	\(\sqrt{1-a^2}\)	&&	\( -a\big|_{a=1} \)	&&	\( \frac 1 2 \arccos(-a)\big|_{a=1} = \frac \pi 2 \) && Relevant for \( a= 1\)
	\(\frac{\pi}{2}\)	&& 	\(0\)	&&	\( -1 \)	&& \(2\log \varepsilon\)	&&	\( \ \frac \pi 2 \) && 
	\\
	\( 2\pi\) && \(1-a\) && \(-1\) && \(2\log |1-a|\) && \( \frac \pi 2\)
	\\
	\hline
	\multicolumn{11}{c}{ \( \sin \alpha = \frac 1 a \) } 
	\\[0.5em]
	\( \alpha \)	&& 	\( -\sqrt{a^2 -1} \)	&&	\( - \frac 1 a \)	&&	\( 2 \log (a^2-1)\) &&  \( \frac{1}{4} (2\alpha + \pi) \) && \( a \geq 1\)
	\\
	\(\pi-\alpha\)	&& 	\( -s(\alpha) \)	&&	\( - \frac 1 a\)	&&	\(P(\alpha)\) &&
	\( \frac{1}{4} ( 2\alpha + \pi) \) %\( (\sfrac 1 2) \arccos(-a^{-1}) = (\sfrac 1 4) ( 2\alpha + \pi) \)
	\\
	\(2\pi - \alpha\)	&& 	\( \hphantom{-}s(\alpha)\)	&&	\(- \frac 1 a\)	&&	\(P(\alpha)\) &&	\(\frac{1}{4} ( 2\alpha + \pi) \)
	\\
	\hline
	\multicolumn{11}{c}{ \(\varphi = \arccos \frac{1 - a^2 - \varepsilon^2}{2a\varepsilon} \) } 
	\\[0.5em]
	\( \varphi\)	 	&& 	\( \varepsilon \)	&&	\( \frac{\varepsilon^2 - 1 - a^2}{2a} \)	&& \( 2\log \varepsilon\) &&	\( \times\) && \( a^2 \leq 1 + \varepsilon^2\)
	\\
	\( \varphi\)	 	&& 	\( \frac{a^2 -1}{\varepsilon} \)	&& \( \frac{(a^2 -1)^2}{2a \varepsilon^2} - \frac{1 + a^2}{2a} \)	&&	\( 2 \log \frac{a^2-1}{\varepsilon} \) &&	\( \times\) && \( a^2 > 1 + \varepsilon^2\)
	\\
	\( \varphi + \pi\) && \(-\frac{a^2-1}{s(\varphi)}\) && \( \frac{ (a^2-1)^2 }{1+a^2 + 2a L(\varphi)} - \frac{1+a^2}{2a} \) && \(2 \log |a^2-1| - P(\varphi) \)  && \(\times\) && \( a\neq 1\)
	\\
	\( 0 \)	 	&& 	\( \varepsilon \)	&&	\( -1 \)	&& \(2 \log \varepsilon \) &&	\( \frac \pi 2\) && \( \varphi|_{a=1-\varepsilon}\)
	\\
	\( %\varphi = 
	\arccos -\frac \varepsilon 2\) && \(\varepsilon\) && \(\frac{1}{2} ( \varepsilon^2 - 2 )\) && \( 2\log \varepsilon\) && \( \times \) && \( \varphi |_{a = 1}\)
	\\
	\( %\varphi = 
	\arccos -\frac{\varepsilon}{ \sqrt{ 1+ \varepsilon^2 } }\)	 	&& 	\( \varepsilon \)	&&	\( -(1+\varepsilon^2)^{-\sfrac 1 2} \)	&& \( 2 \log \varepsilon \) &&	\( \times\) && \( \varphi|_{a^2=1+\varepsilon^2}\)
	\\
	\( \pi\)	 	&& 	\( 2+\varepsilon \)	&&	\( 1 \)	&& \(2 \log(\varepsilon + 2) \) &&	\( 0\) && \( \varphi|_{a=1+\varepsilon}\)
	\\
\bottomrule
\end{tabular}
\\[0.5em]
%\footnotetext
%\flushleft
\caption{Selected values of \(s(\theta) = -a \cos \theta + \sqrt{1 - a^2 \sin^2 \theta}\), \(L(\theta) = \dfrac{s^2(\theta) - 1 - a^2}{2a} \), \(P(\theta) \coloneqq \log\del{1+ a^2 +2a L(\theta) } = 2\log |s(\theta)| \), and \( \Phi(\theta) = \frac{1}{2} \arccos L(\theta)\). 
\newline
\textbf{Notes:} (a) We use the marker \(\times\) when no closed-form expression is available. (b) The values for \( \theta = \sfrac \pi 2\) are only relevant for \(a = 1\). }
\vspace{-0.5\baselineskip}
\label{tb:Relevant_Quantities_and_Values}
\end{table}
% ------------------------------------- %

% ------------------------------------- %
\begin{table}[h!]
\vspace{2\baselineskip}
\setlength{\tabcolsep}{0.2em}
\centering
%\fontsize{6.5}{7.5}\selectfont
%\footnotesize
\fontsize{8}{9}\selectfont
\def\arraystretch{1.95}
\begin{tabular}{c c l c l c l c c c  l}
\toprule
\multirow{2}{*}{ \bf Radius $a$ } 
&&
\multirow{2}{*}{ \bf Angle $\theta$ } 
&&  \multicolumn{4}{c}{\bf Quantities of interest }
& \multirow{2}{*}{ \bf Comments } 
\\  \cline{5-7}
&&&& \( \Phi(\theta) \) && \multicolumn{1}{c}{\( G(a; \Phi(\theta) ) \)}
\\
\midrule
	\(a \) && \(	\{0, \sfrac \pi 2, 2\pi\} \)	&& 	\(\frac \pi 2\)	&&	
	\(  \pi   
	\begin{cases} 
		(1-a^2) & a\leq 1, 
		%\\ 
		%(1-a^2) -2 \ln a & 1 < a \leq \sqrt{1 + \varepsilon^2},
		\\
		-2 \ln a & a > 1, %\sqrt{1 + \varepsilon^2},
	\end{cases}
 \)
 	\\
	\(1 - \varepsilon\) && \(\varphi = 0\) && \(\frac \pi 2\) && \( \pi \varepsilon (2-\varepsilon)\) %\( \pi \del{1-(1-\varepsilon)^2} \) %\(G(1-\varepsilon; \sfrac \pi 2)\)
	\\
	%\(1\) && \(\theta\) && \(\times\) && \( 2 \Im \big( \Li (-e^{2i\theta}) \big)	+ 2\big( 1- \log 2|\cos \theta|  \big)\sin 2\theta \)
	%\\
	%\( 1\) && \( \sfrac{\pi}{2}\)	&& 	\(\frac \pi 2\)	&&	\( 0 \)
	%\\
	\( 1 \) && \(\varphi = \arccos -\frac \varepsilon 2\) && 
	\( \arccos \frac{\varepsilon}{2} \) % \(\frac{1}{2} \arccos \frac{\varepsilon^2 - 2}{2}\)  [Use the cos2x identity]
	&& \( 2 \Im \Li \del{- e^{2i\Phi(\varphi) }} + \varepsilon \sqrt{4 - \varepsilon^2}  (1 - \log \varepsilon) \)
	&& %\( \cos \Phi(\varphi) = \sfrac \varepsilon 2\)
	\\
	\( 1+\varepsilon\) && \(\varphi + \pi = 2\pi\) && \(\frac \pi 2\) && \(2\pi \ln (1+\varepsilon)\) %\(G(1+\varepsilon; \sfrac \pi 2)\)
	\\
	\(a\) && \( \{\alpha, \pi-\alpha, 2\pi-\alpha\} \)	&&  \( \frac{1}{4} (2\alpha + \pi) \) &&
	\(
		%2 \Im \Li \del{ -a e^{i(\alpha - \sfrac \pi 2)} } + (1-a^2) (\alpha + \pi)  - \del{2 - \log(a^2 - 1) } \sqrt{a^2 -1} % \alpha > \pi/2
		2 \Im \Li \del{ -a e^{i(\alpha + \sfrac \pi 2)} } + (1-a^2) \alpha  + \del{2 - \log(a^2 - 1) } \sqrt{a^2 -1}   
	\)
	&& \( \sin \alpha = \sfrac 1 a \leq 1 \)
	\\
\bottomrule
\end{tabular}
\\[0.5em]
%\footnotetext
%\flushleft
\caption{Selected values of \(G(a;\phi) = 2 \Im \big( \Li (-a e^{2i\phi}) \big)
	+ (1-a^2) \left[ 2 \phi - \arctan\left(\frac{ a \sin 2\phi}{1 + a \cos 2\phi} \right)  \right]
	+ a\big( 2- \log(1+ a^2 +2a\cos 2\phi)  \big)\sin 2\phi
	\), for relevant values of \((a,\theta)\) and \(\phi = \Phi(\theta)\).}
\vspace{-0.5\baselineskip}
\label{tb:Values_of_G}
\end{table}
% ------------------------------------- %

\end{appendices}

\end{document}